\documentclass{amsart}

\usepackage{amssymb,latexsym,amsfonts,amsmath}
\usepackage{multicol} 
\usepackage{graphicx}
\include{diagrams}
\usepackage{eso-pic}
\usepackage{epsfig}
\usepackage{graphicx}
\usepackage{color}
\usepackage{tikz}
\usetikzlibrary{automata}
\usetikzlibrary{shapes}
\usepackage[normalem]{ulem}
\usepackage{refcount}

\usepackage{amssymb,latexsym,amsfonts,amsmath}
\usepackage{graphicx}

\newtheorem{theorem}{Theorem}[section]
\newtheorem{lemma}[theorem]{Lemma}

\newtheorem{definition}[theorem]{Definition}

\newtheorem{remark}[theorem]{Remark}
\numberwithin{equation}{section}

\newcommand{\R}{{\mathbb{R}}}

\newcommand{\Ze}{{\mathbb Z}}

\newcommand{\N}{{\mathbb{N}}}

\begin{document}

\begin{abstract}
Incremental stability is a property of dynamical and control systems, requiring the uniform asymptotic stability of every trajectory, rather than that of an equilibrium point or a particular time-varying trajectory. Similarly to stability, Lyapunov functions and contraction metrics play important roles in the study of incremental stability. In this paper, we provide characterizations and descriptions of incremental stability in terms of existence of coordinate-invariant notions of incremental Lyapunov functions and contraction metrics, respectively. Most design techniques providing controllers rendering control systems incrementally stable have two main drawbacks: they can only be applied to control systems in either parametric-strict-feedback or strict-feedback form, and they require these control systems to be smooth. In this paper, we propose a design technique that is applicable to larger classes of (not necessarily smooth) control systems. Moreover, we propose a recursive way of constructing contraction metrics (for smooth control systems) and incremental Lyapunov functions which have been identified as a key tool enabling the construction of finite abstractions of nonlinear control systems, the approximation of stochastic hybrid systems, source-code model checking for nonlinear dynamical systems and so on. The effectiveness of the proposed results in this paper is illustrated by synthesizing a controller rendering a non-smooth control system incrementally stable as well as constructing its finite abstraction, using the computed incremental Lyapunov function.
\end{abstract}

\title[Backstepping controller synthesis and characterizations of incremental stability]{Backstepping controller synthesis and characterizations of incremental stability}
\thanks{This research was sponsored in part by the grant NSF-CNS-0953994 and a contract from Toyota Corp.}

\author[Majid Zamani]{Majid Zamani$^1$} 
\author[Nathan van de Wouw]{Nathan van de Wouw$^2$} 
\author[Rupak Majumdar]{Rupak Majumdar$^3$}
\address{$^1$Department of Electrical Engineering\\
University of California at Los Angeles,
Los Angeles, CA 90095}
\email{zamani@ee.ucla.edu}
\urladdr{http://www.ee.ucla.edu/~zamani}
\address{$^2$Department of Mechanical Engineering, Eindhoven University of Technology, Eindhoven, The Netherlands.}
\email{N.v.d.Wouw@tue.nl}
\urladdr{ http://www.dct.tue.nl/vandewouw}
\address{$^3$Max Planck Institute for Software Systems, Kaiserslautern, Germany}
\email{rupak@mpi-sws.org}
\urladdr{http://www.cs.ucla.edu/~rupak}

\maketitle

\section{Introduction}
Incremental stability is a stronger property than stability for dynamical and control systems. In incremental stability, focus is on convergence of trajectories with respect to each other rather than with respect to an equilibrium point or a specific trajectory. Similarly to stability, Lyapunov functions play an important role in the study of incremental stability. In \cite{angeli}, Angeli proposed the notions of incremental Lyapunov function and incremental input-to-state Lyapunov function, and used these notions to provide characterizations of incremental global asymptotic stability ($\delta$-GAS) and incremental input-to-state stability ($\delta$-ISS). Notions of $\delta$-GAS, $\delta$-ISS and incremental Lyapunov functions, proposed in \cite{angeli}, are not coordinate invariant, in general. Since most of the controller design approaches benefit from changes of coordinates, in \cite{majid1}, the authors proposed different notions of $\delta$-GAS and $\delta$-ISS which are coordinate invariant. In \cite{majid4}, the authors proposed notions of incremental Lyapunov function and incremental input-to-state Lyapunov function that are coordinate invariant as well. We use these new notions of Lyapunov functions to fully characterize the notions of incremental (input-to-state) stability as proposed in \cite{majid1}. Furthermore, we provide sufficient conditions for coordinate invariant incremental (input-to-state) stability in the form of contraction metrics inspired by the work in \cite{aghannan}.

The number of applications of incremental stability has increased progressively in the past years. Examples include building explicit bounds on the region of attraction in phase-locking in the Kuramoto system \cite{franci}, modeling of nonlinear analog circuits \cite{bond}, robustness analysis of systems over finite alphabets \cite{tarraf}, global synchronization in networks of cyclic feedback systems \cite{hamadeh}, control reconfiguration of piecewise affine systems with actuator and sensor faults \cite{richter}, construction of symbolic models for nonlinear control systems \cite{pola,girard2,pola1}, and synchronization \cite{russo,stan}. Unfortunately, there are very few results available in the literature regarding the design of controllers enforcing incremental stability of the resulting closed-loop systems. Therefore, there is a growing need to develop design methods rendering control systems incrementally stable. 

Related works include controller designs for convergence of Lur'e-type systems \cite{pavlov,pavlov2} and a class of piecewise affine systems \cite{wouw} through the solution of linear matrix inequalities (LMIs). In contrast, the current paper does not require the solution of LMIs and the existence of controllers is always guaranteed for the class of systems under consideration. The quest for backstepping design approaches for incremental stability has received increasing attention recently. Recently obtained results include backstepping design approaches rendering parametric-strict-feedback\footnote{\label{kkk}See \cite{miroslav} for a definition.} form systems incrementally globally asymptotically stable\footnote{Understood in the sense of Definition \ref{delta_GAS1}.}  using the notion of contraction metrics in \cite{jouffroy1,sharma,sharma1}, and backstepping design approaches rendering strict-feedback\footnotemark[\getrefnumber{kkk}] form systems incrementally input-to-state stable\footnote{Understood in the sense of Definition \ref{dISS}.} using the notion of contraction metrics and incremental Lyapunov functions in \cite{majid1}, and \cite{majid4}, respectively. The results in \cite{pavlov} offer a backstepping design approach rendering a larger class of control systems than those in strict-feedback form input-to-state convergent, rather than incrementally input-to-state stable. We will build upon these results in \cite{pavlov} and extend those in the scope of incremental stability. The notion of (input-to-state) convergence requires existence of a trajectory which is bounded on the whole time axis which is not required in the case of incremental input-to-state stability. The notion of input-to-state convergence can not be applied to the results in \cite{pola,girard2,pola1},  which require the uniform global asymptotic stability of every trajectory rather than that of a particular trajectory that is bounded on the entire time axis. See \cite{majid1,pavlov1} for a brief comparison between the notions of convergent system and incremental stability.
  
Our techniques improve upon most of the existing backstepping techniques in three directions:
\begin{itemize}
\item[1)] by providing controllers enforcing not only incremental global asymptotic stability but also incremental input-to-state stability; 
\item[2)] by being applicable to larger classes of (non-smooth) control systems;
\item[3)] by providing a recursive way of constructing not only contraction metrics but also incremental Lyapunov functions.
\end{itemize}
In the first direction, our technique extends the results in \cite{jouffroy1,sharma,sharma1}, where only controllers enforcing incremental global asymptotic stability are designed. In the second direction, our technique improves the results in \cite{jouffroy1,sharma,sharma1}, which are only applicable to smooth parametric-strict-feedback form systems, and the results in \cite{majid1,majid4}, which are only applicable to smooth strict-feedback form systems. In the third direction, our technique extends the results in \cite{jouffroy1,sharma,sharma1,majid1}, where the authors only provide a recursive way of constructing contraction metrics, and the results in \cite{pavlov}, where the authors do not provide a way to construct Lyapunov functions characterizing the input-to-state convergence property induced by the controller. Note that the explicit availability of incremental Lyapunov functions is necessary in many applications. Examples include the construction of symbolic models for nonlinear control systems \cite{girard2,girard3,camara}, robust test generation of hybrid systems \cite{julius}, the approximation of stochastic hybrid systems \cite{julius1}, and source-code model checking for nonlinear dynamical systems \cite{kapinski}. Note that incremental Lyapunov functions can be used as bisimulation functions, recognized as a key tool for the analysis in \cite{julius,julius,kapinski}.

Our technical results are illustrated by designing an incrementally input-to-state stabilizing controller for an unstable non-smooth control system that does not satisfy the assumptions required in \cite{jouffroy1,sharma,sharma1,majid1,majid4}. Moreover, we construct a finite bisimilar abstraction for the resulting incrementally stable closed-loop system using the results in \cite{girard2}, which, however, apply only to incrementally stable systems with explicitly available incremental Lyapunov functions. When a finite abstraction is available, the synthesis of the controllers satisfying logic specifications expressed in linear temporal logic or automata on infinite strings can be easily reduced to a fixed-point computation over the finite-state abstraction \cite{paulo}. Note that satisfying those specifications is difficult or impossible to enforce with conventional control design methods. We synthesize another controller on top of the resulting incrementally stable closed-loop system satisfying some logic specification explained in details in the example section. 

The outline of the paper is as follows. Section 2 provides some mathematical preliminaries, the definition of the class of control systems that we consider in this paper, and stability notions. Section 3 provides characterizations of incremental stability in terms of existence of incremental Lyapunov functions and contraction metrics. In Section 4, we present the proposed backstepping design approach. An illustrative (non-smooth) example is discussed in details in Section 5. Finally, Section 6 concludes the paper.   

\section{Control Systems and Stability Notions}
\subsection{Notation} 
The symbols $\Ze$, $\N$, $\N_0$, $\mathbb{R}$, $\mathbb{R}^+$ and $\mathbb{R}_0^+$ denote the set of integer, positive integer, nonnegative integer, real, positive, and nonnegative real numbers, respectively. The symbols $I_m$, $0_{m\times{n}}$, and $0_m$ denote the identity and zero matrices in $\R^{m\times{m}}$ and $\R^{m\times{n}}$, and the zero vector in $\R^m$, respectively. Given a vector $x\in\mathbb{R}^{n}$, we denote by $x_{i}$ the \mbox{$i$--th} element of $x$, by $\vert{x}_i\vert$ the absolute value of $x_i$, and by $\Vert x\Vert$ the Euclidean norm of $x$; we recall that \mbox{$\Vert x\Vert=\sqrt{x_1^2+x_2^2+...+x_n^2}$} for $x\in\R^n$. Given a measurable function \mbox{$f:\mathbb{R}_{0}^{+}\rightarrow\mathbb{R}^n$}, the (essential) supremum of $f$ is denoted by $\Vert f\Vert_{\infty}$; we recall that \mbox{$\Vert f\Vert_{\infty}:=\text{(ess)sup}\{\Vert f(t)\Vert,t\geq0\}$} and $\Vert f\Vert_{[0,\tau)}:=\text{(ess)sup}\{\Vert f(t)\Vert,t\in[0,\tau)\}$. A function $f$ is essentially bounded if $\Vert{f}\Vert_{\infty}<\infty$. For a given time $\tau\in\mathbb{R}^+$, define $f_{\tau}$ so that $f_{\tau}(t)=f(t)$, for any $t\in[0,\tau)$, and $f_{\tau}(t)=0$ elsewhere; $f$ is said to be locally essentially bounded if for any $\tau\in\mathbb{R}^+$, $f_{\tau}$ is essentially bounded. A function $f:\R^n\rightarrow\R^+_0$ is called radially unbounded if $f(x)\rightarrow\infty$ as $\Vert{x}\Vert\rightarrow\infty$. The closed ball centered at $x\in\mathbb{R}^m$ with radius $\varepsilon$ is defined by $\mathcal{B}_\varepsilon(x)=\{y\in\mathbb{R}^m\,\,|\,\,\Vert{x}-y\Vert\leq\varepsilon\}$. A continuous function \mbox{$\gamma:\mathbb{R}_{0}^{+}\rightarrow\mathbb{R}_{0}^{+}$}, is said to belong to class $\mathcal{K}$ if it is strictly increasing and \mbox{$\gamma(0)=0$}; $\gamma$ is said to belong to class $\mathcal{K}_{\infty}$ if \mbox{$\gamma\in\mathcal{K}$} and $\gamma(r)\rightarrow\infty$ as $r\rightarrow\infty$. A continuous function \mbox{$\beta:\mathbb{R}_{0}^{+}\times\mathbb{R}_{0}^{+}\rightarrow\mathbb{R}_{0}^{+}$} is said to belong to class $\mathcal{KL}$ if, for each fixed $s$, the map $\beta(r,s)$ belongs to class $\mathcal{K}_{\infty}$ with respect to $r$ and, for each fixed nonzero $r$, the map $\beta(r,s)$ is decreasing with respect to $s$ and $\beta(r,s)\rightarrow0$ as \mbox{$s\rightarrow\infty$}. If $\phi:\R^n\to \R^n$ is a global diffeomorphism and \mbox{$G:\R^n\to \R^{n\times n}$} is a smooth map, the notation $\phi^*G:\R^n\to \R^{n\times n}$ denotes the smooth map \mbox{$(\phi^*G)(x)=(\frac{\partial \phi}{\partial x})^TG(\phi(x))(\frac{\partial \phi}{\partial x})$}. A Riemannian metric \mbox{$G:\mathbb{R}^n\rightarrow\mathbb{R}^{n\times n}$} is a smooth map on $\mathbb{R}^n$ such that, for any \mbox{$x\in\mathbb{R}^n$}, $G(x)$ is a symmetric positive definite matrix \cite{lee}. For any \mbox{$x\in\mathbb{R}^n$} and smooth functions \mbox{$I,J:\mathbb{R}^n\rightarrow\mathbb{R}^n$}, one can define the scalar function \mbox{$\langle{I},J\rangle_G$} as $I^T(x)G(x)J(x)$. We will still use the notation $\langle{I},J\rangle_{G}$ to denote $I^TGJ$ even if $G$ does not represent a Riemannian metric. A function \mbox{$\mathbf{d}:\R^n\times \R^n\rightarrow\mathbb{R}_{0}^{+}$} is a metric on $\R^n$ if for any $x,y,z\in\R^n$, the following three conditions are satisfied: i) $\mathbf{d}(x,y)=0$ if and only if $x=y$; ii) $\mathbf{d}(x,y)=\mathbf{d}(y,x)$; and iii) $\mathbf{d}(x,z)\leq\mathbf{d}(x,y)+\mathbf{d}(y,z)$. We use the pair $\left(\R^n,\mathbf{d}\right)$ to denote a metric space $\R^n$ equipped with the metric $\mathbf{d}$. We use the notation $\mathbf{d}_G$ to denote the Riemannian distance function provided by the Riemannian metric $G$, as defined for example in \cite{lee}. We refer to the proof of Lemma \ref{lemma20} in the paper for the definition of $\mathbf{d}_G$. For a set $\mathcal{A}\subseteq\R^n$, a metric $\mathbf{d}$, and any $x\in\R^n$, we abuse the notation by using $\mathbf{d}(x,\mathcal{A})$ to denote the point-to-set distance, defined by $\mathbf{d}(x,\mathcal{A})=\inf_{y\in\mathcal{A}}\mathbf{d}(x,y)$. A function $f$ is said to be smooth if it is an infinitely differentiable function of its arguments. Given measurable functions \mbox{$f:\mathbb{R}_{0}^{+}\rightarrow\mathbb{R}^n$} and \mbox{$g:\mathbb{R}_{0}^{+}\rightarrow\mathbb{R}^n$}, we define $\mathbf{d}(f,g)_\infty:=\text{(ess)sup}\{\mathbf{d}(f(t),g(t)),t\geq0\}$ and $\mathbf{d}(f,g)_{[0,\tau)}:=\text{(ess)sup}\{\mathbf{d}(f(t),g(t)),t\in[0,\tau)\}$.

\subsection{Control Systems\label{II.B}}
The class of control systems that we consider in this paper is formalized in
the following definition.
\begin{definition}
\label{Def_control_sys}A \textit{control system} is a quadruple:
\[
\Sigma=(\mathbb{R}^{n},\mathsf{U},\mathcal{U},f),
\]
where:
\begin{itemize}
\item $\mathbb{R}^{n}$ is the state space;
\item $\mathsf{U}\subseteq\mathbb{R}^{m}$ is the input set;
\item $\mathcal{U}$ is the set of all measurable, locally essentially bounded functions of time from intervals of the form \mbox{$]a,b[\subseteq\mathbb{R}$} to $\mathsf{U}$ with $a<0$ and $b>0$; 
\item \mbox{$f:\mathbb{R}^{n}\times \mathsf{U}\rightarrow\mathbb{R}^{n}$} is a continuous map
satisfying the following Lipschitz assumption: for every compact set
\mbox{$Q\subset\mathbb{R}^{n}$}, there exists a constant $Z\in\mathbb{R}^+$ such that $\Vert
f(x,u)-f(y,u)\Vert\leq Z\Vert x-y\Vert$ for all $x,y\in Q$ and all $u\in \mathsf{U}$.
\end{itemize}
\end{definition}

A curve \mbox{$\xi:]a,b[\rightarrow\mathbb{R}^{n}$} is said to be a
\textit{trajectory} of $\Sigma$ if there exists $\upsilon\in\mathcal{U}$
satisfying:
\begin{equation}\nonumber
\dot{\xi}(t)=f\left(\xi(t),\upsilon(t)\right),
\end{equation}
for almost all $t\in$ $]a,b[$. Although we have defined trajectories over open domains, we shall refer to trajectories \mbox{${\xi:}[0,t]\rightarrow\mathbb{R}^{n}$} defined on closed domains $[0,t],$ $t\in\mathbb{R}^{+}$ with the understanding of the existence of a trajectory \mbox{${\xi}^{\prime}:]a,b[\rightarrow\mathbb{R}^{n}$} such that \mbox{${\xi}={\xi}^{\prime}|_{[0,t]}$} with $a<0$ and $b>t$. We also write $\xi_{x\upsilon}(t)$ to denote the point reached at time $t$
under the input $\upsilon$ from initial condition $x=\xi_{x\upsilon}(0)$; the point $\xi_{x\upsilon}(t)$ is
uniquely determined, since the assumptions on $f$ ensure existence and
uniqueness of trajectories \cite{sontag1}. 

A control system $\Sigma$ is said to be forward complete if every trajectory is defined on an interval of the form $]a,\infty[$. Sufficient and necessary conditions for a system to be forward complete can be found in \cite{sontag}. A control system $\Sigma$ is said to be smooth if $f$ is smooth.

\subsection{Stability notions}
Here, we recall the notions of incremental global asymptotic stability ($\delta_\exists$-GAS) and incremental input-to-state stability ($\delta_\exists$-ISS), presented in \cite{majid1}.

\begin{definition}[\cite{majid1}]
\label{delta_GAS1}
A control system $\Sigma=(\mathbb{R}^{n},\mathsf{U},\mathcal{U},f)$ is incrementally globally asymptotically stable ($\delta_\exists$-GAS) if it is forward complete and there exist a metric $\mathbf{d}$ and a $\mathcal{KL}$ function $\beta$ such that for any $t\in{\mathbb{R}_0^+}$, any $x,x'\in{\mathbb{R}^n}$ and any $\upsilon\in\mathcal{U}$ the following condition is satisfied:
\begin{equation}
\mathbf{d}\left(\xi_{x\upsilon}(t),\xi_{x'\upsilon}(t)\right) \leq\beta\left(\mathbf{d}\left(x,x'\right),t\right). \label{delta_GAS}%
\end{equation}
\end{definition}

As defined in~\cite{angeli}, $\delta$-GAS requires the metric $\mathbf{d}$ to be the Euclidean metric. However, Definition~\ref{delta_GAS1} only requires the existence of a metric. We note that while $\delta$-GAS is not generally invariant under changes of coordinates, $\delta_\exists$-GAS is. 
When the origin is an equilibrium point for $\Sigma$, with $\upsilon(t)=0$ for all $t\in\R^+_0$, and the map $\psi:\R^n\to \R_0^+$, defined by $\psi(\cdot)=\mathbf{d}(\cdot,0)$, is continuous\footnote{\label{continuity}Here, continuity is understood with respect to the Euclidean metric.} and radially unbounded, both $\delta_\exists$-GAS and $\delta$-GAS imply global asymptotic stability.

\begin{definition}[\cite{majid1}]
\label{dISS}
A control system $\Sigma=(\mathbb{R}^{n},\mathsf{U},\mathcal{U},f)$ is incrementally input-to-state stable ($\delta_\exists$-ISS) if it is forward complete and there exist a metric $\mathbf{d}$, a $\mathcal{KL}$ function $\beta$, and a $\mathcal{K}_{\infty}$ function $\gamma$ such that for any $t\in{\mathbb{R}_0^+}$, any $x,x'\in{\mathbb{R}^n}$, and any $\upsilon,\upsilon'\in\mathcal{U}$ the following condition is satisfied:
\begin{equation}
\mathbf{d}\left(\xi_{x\upsilon}(t),\xi_{x'{\upsilon}'}(t)\right) \leq\beta\left(\mathbf{d}\left(x,x'\right),t\right)+\gamma\left(\left\Vert{\upsilon}-{\upsilon}'\right\Vert_{\infty}\right). \label{delta_ISS}%
\end{equation}
\end{definition}

By observing (\ref{delta_GAS}) and (\ref{delta_ISS}), it is readily seen that $\delta_\exists$-ISS implies $\delta_\exists$-GAS while the converse is not true in general. Moreover, whenever the metric $\mathbf{d}$ is the Euclidean metric, $\delta_\exists$-ISS becomes $\delta$-ISS as defined in~\cite{angeli}. We note that while $\delta$-ISS is not generally invariant under changes of coordinates, $\delta_\exists$-ISS is. When the origin is an equilibrium point for $\Sigma$, with $\upsilon(t)=0$ for all $t\in\R^+_0$, and the map $\psi:\R^n\to \R_0^+$, defined by $\psi(\cdot)=\mathbf{d}(\cdot,0)$, is continuous\footnotemark[\getrefnumber{continuity}] and radially unbounded, both $\delta_\exists$-ISS and $\delta$-ISS imply input-to-state stability \cite{sontag3}.

\section{Characterizations of Incremental Stability}
This section contains characterizations and descriptions of $\delta_\exists$-GAS and $\delta_\exists$-ISS in terms of existence of incremental Lyapunov functions and contraction metrics, respectively. We note that only the sufficiency part of Lyapunov characterizations of $\delta_\exists$-GAS and $\delta_\exists$-ISS were presented in \cite{majid4}. In Section \ref{back}, we will use such incremental Lyapunov functions and contraction metrics to synthesize controllers rendering closed-loop systems incrementally stable.

\subsection{Incremental Lyapunov function characterizations}
We start by recalling the notions of an incremental global asymptotic stability ($\delta_\exists$-GAS) Lyapunov function and an incremental input-to-state stability ($\delta_\exists$-ISS) Lyapunov function, presented in \cite{majid4}.

\begin{definition}[\cite{majid4}]
\label{delta_GAS_Lya}
Consider a control system $\Sigma=(\mathbb{R}^n,\mathsf{U},\mathcal{U},f)$ and a smooth function \mbox{$V:\mathbb{R}^n\times\mathbb{R}^n\rightarrow\mathbb{R}_0^+$}. Function $V$ is called a $\delta_\exists$-GAS Lyapunov function for $\Sigma$, if there exist a metric $\mathbf{d}$, $\mathcal{K}_{\infty}$ functions $\underline{\alpha}$, $\overline{\alpha}$, and $\kappa\in\mathbb{R}^+$ such that:
\begin{itemize}
\item[(i)] for any $x,x'\in\mathbb{R}^n$,\\
$\underline{\alpha}(\mathbf{d}(x,x'))\leq{V}(x,x')\leq\overline{\alpha}(\mathbf{d}(x,x'))$;
\item[(ii)] for any $x,x'\in\mathbb{R}^n$ and any $u\in\mathsf{U}$,\\
$\frac{\partial{V}}{\partial{x}}f(x,u)+\frac{\partial{V}}{\partial{x'}}f(x',u)\leq -\kappa V(x,x')$.
\end{itemize}
Function $V$ is called a $\delta_\exists$-ISS Lyapunov function for $\Sigma$, if there exist a metric $\mathbf{d}$, $\mathcal{K}_{\infty}$ functions $\underline{\alpha}$, $\overline{\alpha}$, $\sigma$, and $\kappa\in\mathbb{R}^+$ satisfying conditions (i) and:
\begin{itemize}
\item[(iii)] for any $x,x'\in\mathbb{R}^n$ and for any $u,u'\in\mathsf{U}$,\\
\mbox{$\frac{\partial{V}}{\partial{x}}f(x,u)+\frac{\partial{V}}{\partial{x'}}f(x',u')\leq -\kappa V(x,x')+\sigma(\Vert{u}-u'\Vert)$}.
\end{itemize}
\end{definition}

To provide characterizations of $\delta_\exists$-ISS (resp. $\delta_\exists$-GAS) in terms of the existence of $\delta_\exists$-ISS (resp. $\delta_\exists$-GAS) Lyapunov functions, we need the following technical results.

Here, we introduce the following definition which was inspired by the notion of uniform global asymptotic stability (UGAS) with respect to sets, presented in \cite{lin}.
\begin{definition}
A control system $\Sigma=(\mathbb{R}^n,\mathsf{U},\mathcal{U},f)$ is uniformly globally asymptotically stable (U$_\exists$GAS) with respect to a set $\mathcal{A}\subseteq\R^n$ if it is forward complete and there exist a metric $\mathbf{d}$, and a $\mathcal{KL}$ function $\beta$ such that for any $t\in\R_0^+$, any $x\in\R^n$ and any $\upsilon\in\mathcal{U}$, the following condition is satisfied:
\begin{equation}\label{UGAS}
\mathbf{d}(\xi_{x\upsilon}(t),\mathcal{A})\leq\beta(\mathbf{d}(x,\mathcal{A}),t).
\end{equation} 
\end{definition}

We now introduce the following definition which was inspired by the notion of uniform global asymptotic stability (UGAS) Lyapunov functions in \cite{lin}.
\begin{definition}\label{UGAS_Lya}
Consider a control system $\Sigma=(\mathbb{R}^n,\mathsf{U},\mathcal{U},f)$, a set $\mathcal{A}\subseteq\R^n$, and a smooth function $V:\R^n\rightarrow\R_0^+$. Function $V$ is called a U$_\exists$GAS Lyapunov function, with respect to $\mathcal{A}$, for $\Sigma$, if there exist a metric $\mathbf{d}$, $\mathcal{K}_\infty$ functions $\underline\alpha$, $\overline\alpha$, and $\kappa\in\R^+$ such that:
\begin{itemize}
\item[(i)] for any $x\in\R^n$,\\
$\underline{\alpha}(\mathbf{d}(x,\mathcal{A}))\leq{V}(x)\leq\overline{\alpha}(\mathbf{d}(x,\mathcal{A}))$;
\item[(ii)] for any $x\in\R^n$ and any $u\in\mathsf{U}$,\\
$\frac{\partial{V}}{\partial{x}}f(x,u)\leq -\kappa V(x)$.
\end{itemize}
\end{definition}

The following theorem characterizes U$_\exists$GAS in terms of the existence of a U$_\exists$GAS Lyapunov function.
\begin{theorem}\label{theorem10}
Consider a control system $\Sigma=(\mathbb{R}^n,\mathsf{U},\mathcal{U},f)$ and a set $\mathcal{A}\subseteq\R^n$. If $\mathsf{U}$ is compact and $\mathbf{d}$ is a metric such that the function $\psi(\cdot)=\mathbf{d}(\cdot,y)$ is continuous\footnotemark[\getrefnumber{continuity}] for any $y\in\R^n$ then the following statements are equivalent: 
\begin{itemize}
\item[(1)] $\Sigma$ is forward complete and there exists a U$_\exists$GAS Lyapunov function with respect to $\mathcal{A}$, equipped with the metric ${\mathbf{d}}$.
\item[(2)] $\Sigma$ is U$_\exists$GAS with respect to $\mathcal{A}$, equipped with the metric $\mathbf{d}$.
\end{itemize}
\end{theorem}

\begin{proof}
First we show that the function $\phi(\cdot)=\mathbf{d}(\cdot,\mathcal{A})$ is a continuous function with respect to the Euclidean metric. Assume $\{x_n\}_{n=1}^\infty$ is a converging sequence in $\R^n$ with respect to the Euclidean metric, implying that $\Vert x_n-x^*\Vert\rightarrow0$ as $n\rightarrow\infty$ for some $x^*\in\R^n$. By the triangle inequality, we have:
\begin{equation}\label{ineq1}
\mathbf{d}\left(x^*,y\right)\leq\mathbf{d}\left(x^*,x_n\right)+\mathbf{d}\left(y,x_n\right),
\end{equation}
for any $n\in\N$ and any $y\in\mathcal{A}$. Using inequality (\ref{ineq1}), we obtain:
\begin{align}
\label{ineq2}
\phi\left(x^*\right)=\mathbf{d}(x^*,\mathcal{A})=\inf_{y\in\mathcal{A}}\mathbf{d}\left(x^*,y\right)&\leq\inf_{y\in\mathcal{A}}\left\{\mathbf{d}\left(x^*,x_n\right)+\mathbf{d}\left(y,x_n\right)\right\}\\\notag&=\mathbf{d}\left(x^*,x_n\right)+\inf_{y\in\mathcal{A}}\mathbf{d}\left(y,x_n\right)=\mathbf{d}\left(x^*,x_n\right)+\phi\left(x_n\right),
\end{align}
for any $n\in\N$. Using inequality (\ref{ineq2}) and the continuity assumption on $\mathbf{d}$, implying that $\lim_{n\rightarrow\infty}\mathbf{d}\left(x^*,x_n\right)=\mathbf{d}(x^*,x^*)=0$, we obtain for any $n\in\N$:
\begin{equation}\label{ineq3}
\phi\left(x^*\right)\leq\inf_{m\geq{n}}\left\{\mathbf{d}\left(x^*,x_m\right)+\phi\left(x_m\right)\right\}\Rightarrow\phi\left(x^*\right)\leq\lim_{n\rightarrow\infty}\inf\phi\left(x_n\right),
\end{equation}
where limit inferior exists because a lower bounded sequence of real numbers always admit a greatest lower bound \cite{radulescu}. By doing the same analysis, we have:
\begin{equation}\label{ineq4}
\phi\left(x^*\right)\geq\lim_{n\rightarrow\infty}\sup\phi\left(x_n\right),
\end{equation}
where limit superior exists because an upper bounded sequence of real numbers always admit a lowest upper bound \cite{radulescu}. Using inequalities (\ref{ineq3}) and (\ref{ineq4}), one obtains:
\begin{equation}\nonumber
\phi\left(x^*\right)=\lim_{n\rightarrow\infty}\phi\left(x_n\right),
\end{equation}
implying that $\phi$ is a continuous function with respect to the Euclidean metric. Since $\phi(\cdot)=\mathbf{d}(\cdot,\mathcal{A})$ is a continuous, positive semi-definite function, by choosing \mbox{$\omega_1(\cdot)=\omega_2(\cdot)=\mathbf{d}(\cdot,\mathcal{A})$} in Theorem 1 in \cite{teel}, the proof completes.
\end{proof}

Before showing the main results, we need the following technical lemma, inspired by Lemma 2.3 in \cite{angeli}.

\begin{lemma}\label{lemma10}
Consider a control system \mbox{$\Sigma=(\mathbb{R}^n,\mathsf{U},\mathcal{U},f)$}. If $\Sigma$ is $\delta_\exists$-GAS, then the control system $\widehat\Sigma=(\R^{2n},\mathsf{U},\mathcal{U},\widehat{f})$, where $\widehat{f}(\zeta,\upsilon)=\left[f(\xi_1,\upsilon)^T,f(\xi_2,\upsilon)^T\right]^T$, and $\zeta=\left[\xi_1^T,\xi_2^T\right]^T$, is U$_\exists$GAS with respect to the diagonal set $\Delta$, defined by:
\begin{equation}\label{delta}
\Delta=\left\{z\in\R^{2n}\\|\\\exists x\in\R^n:z=\left[x^T,x^T\right]^T\right\}.
\end{equation}
\end{lemma}

\begin{proof}
Since $\Sigma$ is $\delta_\exists$-GAS, there exists a metric \mbox{$\mathbf{d}:\R^n\times\R^n\rightarrow\R_0^+$} such that property (\ref{delta_GAS}) is satisfied. Now we define a new metric $\widehat{\mathbf{d}}:\R^{2n}\times\R^{2n}\rightarrow\R^+_0$ by: 
\begin{equation}\label{metric1}
\widehat{\mathbf{d}}(z,z')=\mathbf{d}(x_1,x'_1)+\mathbf{d}(x_2,x'_2),
\end{equation}
for any $z=\left[{x_1}^T,{x_2}^T\right]^T$, $z'=\left[{x'_1}^T,{x'_2}^T\right]^T\in\R^{2n}$. It can be readily checked that $\widehat{\mathbf{d}}$ satisfies all three conditions of a metric. Now we show that $\widehat{\mathbf{d}}(z,\Delta)$, for any $z=\left[x_1^T,x_2^T\right]^T\in\R^{2n}$, is proportional to $\mathbf{d}(x_1,x_2)$ that will be exploited later in the proof. We have:
\begin{align}
\label{ineq5}
\widehat{\mathbf{d}}(z,\Delta)=\inf_{z'\in\Delta}\widehat{\mathbf{d}}(z,z')&=\inf_{x'\in\R^n}\widehat{\mathbf{d}}\left(\left[{\begin{array}{c}x_1\\x_2\\\end{array}}\right],\left[{\begin{array}{c}x'\\x'\\\end{array}}\right]\right)=\inf_{x'\in\R^n}\left\{\mathbf{d}(x_1,x')+\mathbf{d}(x_2,x')\right\}\\\notag&\leq\inf_{x'=x_1}\left\{\mathbf{d}(x_1,x')+\mathbf{d}(x_2,x')\right\}=\mathbf{d}(x_1,x_1)+\mathbf{d}(x_1,x_2)=\mathbf{d}(x_1,x_2).
\end{align}
Since $\mathbf{d}$ is a metric, by using the triangle inequality, we have: $\mathbf{d}(x_1,x_2)\leq\mathbf{d}(x_1,x')+\mathbf{d}(x_2,x')$ for any $x'\in\R^n$, implying that: 
\begin{equation}\label{ineq10}
\mathbf{d}(x_1,x_2)\leq\inf_{x'\in\R^n}\left\{\mathbf{d}(x_1,x')+\mathbf{d}(x_2,x')\right\}=\widehat{\mathbf{d}}(z,\Delta).
\end{equation}
Hence, using (\ref{ineq5}) and (\ref{ineq10}), one obtains:
\begin{equation}\label{equality}
\mathbf{d}(x_1,x_2)\leq\widehat{\mathbf{d}}(z,\Delta)\leq\mathbf{d}(x_1,x_2)\Rightarrow\mathbf{d}(x_1,x_2)=\widehat{\mathbf{d}}(z,\Delta).
\end{equation} 
Using equality (\ref{equality}) and property (\ref{delta_GAS}), we have:
\begin{align}\nonumber
\widehat{\mathbf{d}}\left(\zeta_{z\upsilon}(t),\Delta\right)=\widehat{\mathbf{d}}\left(\left[{\begin{array}{c}\xi_{x_1\upsilon}(t)\\\xi_{x_2\upsilon}(t)\\\end{array}}\right],\Delta\right)=\mathbf{d}\left(\xi_{x_1\upsilon}(t),\xi_{x_2\upsilon}(t)\right)\leq\beta\left(\mathbf{d}\left(x_1,x_2\right),t\right)=\beta\left(\widehat{\mathbf{d}}\left(z,\Delta\right),t\right),
\end{align}
for any $t\in\R_0^+$, any $z=\left[x_1^T,x_2^T\right]^T\in\R^{2n}$ and any $\upsilon\in\mathcal{U}$, where $\zeta_{z\upsilon}=\left[\xi_{x_1\upsilon}^T,\xi_{x_2\upsilon}^T\right]^T$. Hence, $\widehat\Sigma$ is U$_\exists$GAS with respect to $\Delta$.
\end{proof}

We can now provide characterization of $\delta_\exists$-GAS in terms of existence of a $\delta_\exists$-GAS Lyapunov function.
\begin{theorem}
Consider a control system \mbox{$\Sigma=(\mathbb{R}^n,\mathsf{U},\mathcal{U},f)$}. If $\mathsf{U}$ is compact and $\mathbf{d}$ is a metric such that the function $\psi(\cdot)=\mathbf{d}(\cdot,y)$ is continuous\footnotemark[\getrefnumber{continuity}] for any $y\in\R^n$ then the following statements are equivalent: 
\begin{itemize}
\item[(1)] $\Sigma$ is forward complete and there exists a $\delta_\exists$-GAS Lyapunov function, equipped with the metric ${\mathbf{d}}$.
\item[(2)] $\Sigma$ is $\delta_\exists$-GAS, equipped with the metric $\mathbf{d}$.
\end{itemize} 
\end{theorem}
\begin{proof}
The proof from (1) to (2) has been provided in Theorem 2.6 in \cite{majid4}, even in the absence of the compactness and continuity assumptions on $\mathsf{U}$ and $\mathbf{d}$, respectively. We now prove that (2) implies (1). Since $\Sigma$ is $\delta_\exists$-GAS, using Lemma \ref{lemma10}, we conclude that the control system $\widehat\Sigma$, defined in Lemma \ref{lemma10}, is U$_\exists$GAS with respect to the diagonal set $\Delta$. Since $\psi(\cdot)=\mathbf{d}(\cdot,y)$ is continuous\footnotemark[\getrefnumber{continuity}]  for any $y\in\R^n$, it can be easily verified that the function $\widehat\psi(\cdot)=\widehat{\mathbf{d}}(\cdot,z')$ is also continuous\footnotemark[\getrefnumber{continuity}]  for any $z'\in\R^{2n}$, where the metric $\widehat{\mathbf{d}}$ was defined in Lemma \ref{lemma10}. Using Theorem \ref{theorem10}, we conclude that there exists a U$_\exists$GAS Lyapunov function $V:\R^{2n}\rightarrow\R_0^+$, with respect to $\Delta$, for $\widehat\Sigma$.
Thanks to the special form of $\widehat\Sigma$, using the equality (\ref{equality}), and slightly abusing notation, the function $V$ satisfies: 
\begin{itemize}
\item[(i)] $\underline\alpha\left(\widehat{\mathbf{d}}\left(\left[{\begin{array}{c}x\\x'\\\end{array}}\right],\Delta\right)\right)\leq V\left(\left[{\begin{array}{c}x\\x'\\\end{array}}\right]\right)\leq\overline\alpha\left(\widehat{\mathbf{d}}\left(\left[{\begin{array}{c}x\\x'\\\end{array}}\right],\Delta\right)\right)\Rightarrow\underline{\alpha}(\mathbf{d}(x,x'))\leq{V}(x,x')\leq\overline{\alpha}(\mathbf{d}(x,x'))$;
\item[(ii)] $\left[\frac{\partial{V}}{\partial{x}}~\frac{\partial{V}}{\partial{x'}}\right]\left[{\begin{array}{c}f(x,u)\\f(x',u)\\\end{array}}\right]\leq-\kappa{V}\left(\left[{\begin{array}{c}x\\x'\\\end{array}}\right]\right)\Rightarrow\frac{\partial{V}}{\partial{x}}f(x,u)+\frac{\partial{V}}{\partial{x'}}f(x',u)\leq -\kappa V(x,x')$,
\end{itemize}
for any $x,x'\in\R^n$, any $u\in\mathsf{U}$, some $\mathcal{K}_\infty$ functions $\underline\alpha,\overline\alpha$ and some $\kappa\in\R^+$. Hence, V is a $\delta_\exists$-GAS Lyapunov function for $\Sigma$. This completes the proof.
\end{proof}

Before providing characterization of $\delta_\exists$-ISS in terms of existence of a $\delta_\exists$-ISS Lyapunov function, we need the following technical lemma, inspired by Proposition 5.3 in \cite{angeli}. By following similar steps as in \cite{angeli}, we need to define the proximal point function $\mathsf{sat}_\mathsf{U}:\R^m\rightarrow\mathsf{U}$, defined by:
\begin{equation}\label{function}
\mathsf{sat}_\mathsf{U}(u)=\arg\min_{u'\in\mathsf{U}}\left\Vert u'-u\right\Vert.
\end{equation} 
As explained in \cite{angeli}, by assuming $\mathsf{U}$ is closed and convex and since $\Vert\cdot\Vert:\R^m\rightarrow\R^+_0$ is a proper and convex function, the definition (\ref{function}) is well-defined and the minimizer of $\left\Vert u'-u\right\Vert$ with $u'\in\mathsf{U}$ is unique. Moreover, by convexity of $\mathsf{U}$ we have:
\begin{equation}\label{convex}
\Vert\mathsf{sat}_\mathsf{U}(u_1)-\mathsf{sat}_\mathsf{U}(u_2)\Vert\leq\Vert u_1-u_2\Vert,~~~~\forall u_1,u_2\in\R^m.
\end{equation}

\begin{lemma}\label{lemma200}
Consider a control system \mbox{$\Sigma=(\mathbb{R}^n,\mathsf{U},\mathcal{U},f)$}, where $\mathsf{U}$ is closed and convex. If $\Sigma$ is $\delta_\exists$-ISS, equipped with a metric $\mathbf{d}$ such that $\psi(\cdot)=\mathbf{d}(\cdot,y)$ is continuous\footnotemark[\getrefnumber{continuity}] for any $y\in\R^n$, then there exists a $\mathcal{K}_\infty$ function $\rho$ such that the control system $\widehat\Sigma=(\R^{2n},\mathsf{D},\mathcal{D},\widehat{f})$\footnote{$\mathcal{D}$ is the set of all measurable and locally essentially bounded functions of time from intervals of the form \mbox{$]a,b[\subseteq\mathbb{R}$} to $\mathsf{D}$ with $a<0$ and $b>0$.} is U$_\exists$GAS with respect to the diagonal set $\Delta$, where: 
\begin{equation}\nonumber
\widehat{f}(\zeta,\omega)=\left[
                \begin{array}{c}
                 f(\xi_1,\mathsf{sat}_\mathsf{U}(\omega_1+\rho(\mathbf{d}(\xi_1,\xi_2))\omega_2))\\ 
                 f(\xi_2,\mathsf{sat}_\mathsf{U}(\omega_1-\rho(\mathbf{d}(\xi_1,\xi_2))\omega_2))
                \end{array}
                \right],
\end{equation}
$\zeta=\left[\xi_1^T,\xi_2^T\right]^T$, $\mathsf{D}=\mathsf{U}\times\mathcal{B}_1(0_m)$, and $\omega=\left[\omega_1^T,\omega_2^T\right]^T$.
\end{lemma} 

\begin{proof}
The proof was inspired by the proof of Proposition 5.3 in \cite{angeli}. We include the complete details of the proof to ensure that the interested reader can assess the essential differences caused by using the arbitrary metric $\mathbf{d}$ rather than the Euclidean metric. Since $\Sigma$ is $\delta_\exists$-ISS, equipped with the metric $\mathbf{d}$, there exist some $\mathcal{KL}$ function $\beta$ and $\mathcal{K}_\infty$ function $\gamma$ such that:
\begin{equation}\label{delta-ISS}
\mathbf{d}(\xi_{x\upsilon}(t),\xi_{x'\upsilon'}(t))\leq\max\{\beta(\mathbf{d}(x,x'),t),\gamma(\Vert\upsilon-\upsilon'\Vert_\infty)\}.
\end{equation} 
Note that inequality (\ref{delta-ISS}) is a straightforward consequence of inequality (\ref{delta_ISS}) in Definition \ref{dISS} (see Remark 2.5 in \cite{sontag3}).
Using the results in Lemma \ref{lemma10} and the proposed metric $\widehat{\mathbf{d}}$ in (\ref{metric1}), we have that $\mathbf{d}(x,x')=\widehat{\mathbf{d}}(z,\Delta)$, where $z=\left[x^T,x'^T\right]^T$, and $\Delta$ was defined in (\ref{delta}). Without loss of generality we can assume $\alpha(r)=\beta(r,0)>r$ for any $r\in\R^+$. Let $\rho$ be a $\mathcal{K}_\infty$ function satisfying $\rho(r)\leq\frac{1}{2}\gamma^{-1}\circ\left(\alpha^{-1}(r)/4\right)$ (note that $\gamma,\alpha\in\mathcal{K}_\infty$). Now we show that 
\begin{equation}\label{inequality0}
\gamma\left(\left\Vert2\omega_2(t)\rho\left(\widehat{\mathbf{d}}(\zeta_{z\omega}(t),\Delta)\right)\right\Vert\right)\leq\widehat{\mathbf{d}}(z,\Delta)/2,
\end{equation}
for any $t\in\R_0^+$, any $z\in\R^{2n}$, and any $\omega=\left[\omega_1^T,\omega_2^T\right]^T\in\mathcal{D}$. Since $\gamma$ is a $\mathcal{K}_\infty$ function and $\omega_2(t)\in\mathcal{B}_1(0_m)$ for any $t\in\R_0^+$, it is enough to show 
\begin{equation}\label{inequality1}
\gamma\left(2\rho\left(\widehat{\mathbf{d}}(\zeta_{z\omega}(t),\Delta)\right)\right)\leq\widehat{\mathbf{d}}(z,\Delta)/2.
\end{equation}
Since 
\begin{align}\nonumber
\gamma\left(2\rho\left(\widehat{\mathbf{d}}(\zeta_{z\omega}(0),\Delta)\right)\right)=\gamma\left(2\rho\left(\widehat{\mathbf{d}}(z,\Delta)\right)\right)\leq\alpha^{-1}\left(\widehat{\mathbf{d}}(z,\Delta)\right)/4<\widehat{\mathbf{d}}(z,\Delta)/4,
\end{align}
and $\varphi(\cdot)=\widehat{\mathbf{d}}(\cdot,\Delta)$ is a continuous\footnotemark[\getrefnumber{continuity}] function (see proof of Theorem \ref{theorem10}), then for all $t\in\R_0^+$ small enough, we have $\gamma\left(2\rho\left(\widehat{\mathbf{d}}(\zeta_{z\omega}(t),\Delta)\right)\right)\leq\widehat{\mathbf{d}}(z,\Delta)/4$. Now, let 
\begin{equation}\nonumber
t_1=\inf\left\{t>0\,\,|\,\,\gamma\left(2\rho\left(\widehat{\mathbf{d}}(\zeta_{z\omega}(t),\Delta)\right)\right)>\widehat{\mathbf{d}}(z,\Delta)/2\right\}.
\end{equation}
Clearly $t_1>0$. We will show that $t_1=\infty$. Now, assume by contradiction that $t_1<\infty$. Therefore, the inequality (\ref{inequality1}) holds for all $t\in[0,t_1)$. Hence, for all $t\in[0,t_1)$, one obtains:
\begin{eqnarray}\label{inequality2}
\gamma\left(\left\Vert2\omega_2(t)\rho\left(\widehat{\mathbf{d}}(\zeta_{z\omega}(t),\Delta)\right)\right\Vert\right)\leq\gamma\left(2\rho\left(\widehat{\mathbf{d}}(\zeta_{z\omega}(t),\Delta)\right)\right)\leq\widehat{\mathbf{d}}(z,\Delta)/2<\alpha\left(\widehat{\mathbf{d}}(z,\Delta)\right)/2.
\end{eqnarray} 
Let $\upsilon$ and $\upsilon'$ be defined as:
\begin{align}
\nonumber
\upsilon(t)&=\mathsf{sat}_\mathsf{U}\left(\omega_1(t)+\rho\left(\widehat{\mathbf{d}}(\zeta_{z\omega}(t),\Delta)\right)\omega_2(t)\right),\\\notag
\upsilon'(t)&=\mathsf{sat}_\mathsf{U}\left(\omega_1(t)-\rho\left(\widehat{\mathbf{d}}(\zeta_{z\omega}(t),\Delta)\right)\omega_2(t)\right).
\end{align}
By using (\ref{convex}), we obtain: $$\Vert \upsilon(t)-\upsilon'(t)\Vert\leq\left\Vert2\omega_2(t)\rho\left(\widehat{\mathbf{d}}(\zeta_{z\omega}(t),\Delta)\right)\right\Vert.$$ Using (\ref{delta-ISS}) and (\ref{inequality2}), we have: 
\begin{eqnarray}\label{ineq6}
\widehat{\mathbf{d}}(\zeta_{z\omega}(t),\Delta)=\mathbf{d}\left(\xi_{x\upsilon}(t),\xi_{x'\upsilon'}(t)\right)\leq\beta\left(\mathbf{d}(x,x'),0\right)=\beta\left(\widehat{\mathbf{d}}(z,\Delta),0\right)=\alpha\left(\widehat{\mathbf{d}}(z,\Delta)\right), 
\end{eqnarray}
for any $t\in[0,t_1)$, any $\omega\in\mathcal{D}$, and any $z=\left[x^T,x'^T\right]^T\in\R^{2n}$. Using $\rho(r)\leq\frac{1}{2}\gamma^{-1}\circ\left(\alpha^{-1}(r)/4\right)$, the inequality (\ref{ineq6}) implies that 
\begin{equation}\label{contradict}
\gamma\left(2\rho\left(\widehat{\mathbf{d}}(\zeta_{z\omega}(t),\Delta)\right)\right)\leq\widehat{\mathbf{d}}(z,\Delta)/4,
\end{equation}
for any $t\in[0,t_1)$. Since the function $\psi(\cdot)=\widehat{\mathbf{d}}(\cdot,\Delta)$ is continuous\footnotemark[\getrefnumber{continuity}], the inequality (\ref{contradict}) contradicts the definition of $t_1$. Therefore, $t_1=\infty$ and inequality (\ref{inequality0}) is proved for all $t\in\R_0^+$. Therefore, using (\ref{delta-ISS}) and (\ref{inequality0}), we obtain:
\begin{align}\nonumber
\widehat{\mathbf{d}}(\zeta_{z\omega}(t),\Delta)=\mathbf{d}\left(\xi_{x\upsilon}(t),\xi_{x'\upsilon'}(t)\right)\leq\max\left\{\beta\left(\mathbf{d}(x,x'),t\right),\gamma\left(\Vert\upsilon-\upsilon'\Vert_\infty\right)\right\}\leq\\\notag\max\left\{\beta\left(\mathbf{d}(x,x'),t\right),\gamma\left(\left\Vert2\omega_2\rho\left(\widehat{\mathbf{d}}(\zeta_{z\omega},\Delta)\right)\right\Vert_\infty\right)\right\}\leq\max\left\{\beta\left(\widehat{\mathbf{d}}(z,\Delta),t\right),\widehat{\mathbf{d}}(z,\Delta)/2\right\},
\end{align}
for any $z=\left[x^T,x'^T\right]^T\in\R^{2n}$, any $\omega\in\mathcal{D}$, and any $t\in\R_0^+$. Since $\beta$ is a $\mathcal{KL}$ function, it can be readily seen that for each $r>0$ if $\widehat{\mathbf{d}}(z,\Delta)\leq r$, then there exists some $T_r\geq0$ such that for any $t\geq T_r$, $\beta\left(\widehat{\mathbf{d}}(z,\Delta),t\right)\leq r/2$ and, hence, $\widehat{\mathbf{d}}(\zeta_{z\omega}(t),\Delta)\leq r/2$. Now we show that the set $\Delta$ is a global attractor for the control system $\widehat\Sigma$. For any $\varepsilon\in\R^+$, let $k$ be a positive integer such that $2^{-k}r<\varepsilon$. Let $r_1=r$ and $r_i=r_{i-1}/2$ for $i\geq2$, and let $\tau=T_{r_1}+T_{r_2}+\cdots+T_{r_k}$. Then, for $t\geq\tau$, we have $\widehat{\mathbf{d}}(\zeta_{z\omega}(t),\Delta)\leq2^{-k}r<\varepsilon$ for all $\widehat{\mathbf{d}}(z,\Delta)\leq r$, and all $\omega\in\mathcal{D}$. Therefore, it can be concluded that the set $\Delta$ is a uniform global attractor for the control system $\widehat\Sigma$. Furthermore, since $\widehat{\mathbf{d}}(\zeta_{z\omega}(t),\Delta)\leq\beta\left(\widehat{\mathbf{d}}(z,\Delta),0\right)$ for all $t\in\R_0^+$, all $z\in\R^{2n}$, and all $\omega\in\mathcal{D}$, the control system $\widehat\Sigma$ is uniformly globally stable and as shown in \cite{teel}, it is U$_\exists$GAS.
\end{proof}

Finally, the next theorem provide a characterization of $\delta_\exists$-ISS in terms of the existence of a $\delta_\exists$-ISS Lyapunov function.
\begin{theorem}
Consider a control system \mbox{$\Sigma=(\mathbb{R}^n,\mathsf{U},\mathcal{U},f)$}. If $\mathsf{U}$ is compact and convex and $\mathbf{d}$ is a metric such that the function $\psi(\cdot)=\mathbf{d}(\cdot,y)$ is continuous\footnotemark[\getrefnumber{continuity}] for any $y\in\R^n$ then the following statements are equivalent: 
\begin{itemize}
\item[(1)] $\Sigma$ is forward complete and there exists a $\delta_\exists$-ISS Lyapunov function, equipped with metric ${\mathbf{d}}$.
\item[(2)] $\Sigma$ is $\delta_\exists$-ISS, equipped with metric $\mathbf{d}$.
\end{itemize}
\end{theorem}
\begin{proof}
The proof from (1) to (2) has been showen in Theorem 2.6 in \cite{majid4}, even in the absence of the compactness and convexity assumptions on $\mathsf{U}$ and the continuity assumption on $\mathbf{d}$. We now prove that (2) implies (1). As we proved in Lemma \ref{lemma200}, since $\Sigma$ is $\delta_\exists$-ISS, it implies that the control system $\widehat\Sigma$, defined in Lemma \ref{lemma200}, is U$_\exists$GAS with respect to $\Delta$. Since $\psi(\cdot)=\mathbf{d}(\cdot,y)$ is continuous\footnotemark[\getrefnumber{continuity}] for any $y\in\R^n$, it can be easily verified that $\widehat\psi(\cdot)=\widehat{\mathbf{d}}(\cdot,z')$ is continuous\footnotemark[\getrefnumber{continuity}] for any $z'\in\R^{2n}$, where the metric $\widehat{\mathbf{d}}$ was defined in the proof of Lemma \ref{lemma10}. Using Theorem \ref{theorem10}, we conclude that there exists a U$_\exists$GAS Lyapunov function V, with respect to $\Delta$, for $\widehat\Sigma$. By using the special form of $\widehat\Sigma$, defined in Lemma \ref{lemma200}, the equality (\ref{equality}), and slightly abusing notation the function $V$ satisfies: 
\begin{itemize}
\item[(i)] $\underline\alpha\left(\widehat{\mathbf{d}}\left(\left[{\begin{array}{c}x\\x'\\\end{array}}\right],\Delta\right)\right)\leq V\left(\left[{\begin{array}{c}x\\x'\\\end{array}}\right]\right)\leq\overline\alpha\left(\widehat{\mathbf{d}}\left(\left[{\begin{array}{c}x\\x'\\\end{array}}\right],\Delta\right)\right)\Rightarrow\underline{\alpha}(\mathbf{d}(x,x'))\leq{V}(x,x')\leq\overline{\alpha}(\mathbf{d}(x,x'))$;
\end{itemize}
for any $x,x'\in\R^n$, some $\mathcal{K}_\infty$ functions $\underline\alpha,\overline\alpha$ and
\begin{itemize}
\item[(ii)] \begin{align}\label{deriv}
&\left[\frac{\partial{V}}{\partial{x}}~\frac{\partial{V}}{\partial{x'}}\right]\left[{\begin{array}{c}f(x,\mathsf{sat}_\mathsf{U}(d_1+\rho(\mathbf{d}(x,x'))d_2))\\f(x',\mathsf{sat}_\mathsf{U}(d_1-\rho(\mathbf{d}(x,x'))d_2))\\\end{array}}\right]\leq-\kappa{V}\left(\left[{\begin{array}{c}x\\x'\\\end{array}}\right]\right)\\\notag
&\Rightarrow\frac{\partial{V}}{\partial{x}}f(x,\mathsf{sat}_\mathsf{U}(d_1+\rho(\mathbf{d}(x,x'))d_2))+\frac{\partial{V}}{\partial{x'}}f(x',\mathsf{sat}_\mathsf{U}(d_1-\rho(\mathbf{d}(x,x'))d_2))\leq-\kappa V(x,x'),
\end{align}
\end{itemize}
for some $\kappa\in\R^+$, any $x,x'\in\R^n$, and any $\left[d_1^T,d_2^T\right]^T\in\mathsf{D}$. By choosing \mbox{$d_1=(u+u')/2$} and $d_2=(u-u')/(2\rho(\mathbf{d}(x,x')))$ for any $u,u'\in\mathsf{U}$, it can be readily checked that \mbox{$\left[d_1^T,d_2^T\right]^T\in\mathsf{U}\times\mathcal{B}_1(0_m)$}, whenever $2\rho(\mathbf{d}(x,x'))\geq\Vert u-u'\Vert$. Hence, using (\ref{deriv}), we have that the following implication holds:
\begin{align}\label{cond2}
\text{if}~\varphi(\mathbf{d}(x,x'))\geq\Vert u-u'\Vert,~\text{then}~\frac{\partial{V}}{\partial{x}}f(x,u)+\frac{\partial{V}}{\partial{x'}}f(x',u')\leq-\kappa V(x,x'),
\end{align}
where $\varphi(r)=2\rho(r)$. As shown in Remark 2.4 in \cite{sontag3}, there is no loss of generality in modifying inequalities (\ref{cond2}) to
\begin{equation}\nonumber
\frac{\partial{V}}{\partial{x}}f(x,u)+\frac{\partial{V}}{\partial{x'}}f(x',u')\leq-\widehat\kappa V(x,x')+\sigma(\Vert u-u'\Vert),
\end{equation}
for some $\mathcal{K}_\infty$ function $\sigma$ and some $\widehat\kappa\in\R^+$, which completes the proof.
\end{proof}

\subsection{Contraction metrics description}
In addition to incremental Lyapunov functions, the $\delta_\exists$-GAS and $\delta_\exists$-ISS conditions can be checked by resorting to contraction metrics. The interested reader may consult \cite{lohmiller} for more detailed information about the notion of contraction metrics. Note that for all definitions and results in this subsection we require function $f$ to be continuously differentiable which was not the case in characterizations of incremental stability using incremental Lyapunov functions. 

Now we recall the notions of contraction metrics, presented in \cite{lohmiller,majid1}.
\begin{definition}[\cite{lohmiller}]
\label{contraction}
Let $\Sigma=(\mathbb{R}^n,\mathsf{U},\mathcal{U},f)$ be a smooth control system on $\R^n$ equipped with a Riemannian metric $G$. The metric $G$ is said to be a contraction metric, with respect to states, for system $\Sigma$ if there exists some $\widehat\lambda\in\mathbb{R}^+$ such that:
\begin{equation}
\langle X,X \rangle_{F}\leq-\widehat\lambda\langle X, X\rangle_{G}, \label{contractionGAS1}
\end{equation}
for $F(x,u)=\left(\frac{\partial{f}}{\partial{x}}\right)^TG(x)+G(x)\frac{\partial{f}}{\partial{x}}+\frac{\partial{G}}{\partial{x}}f(x,u)$, any \mbox{$u\in\mathsf{U}$}, $X\in\mathbb{R}^n$, and $x\in \R^n$. The constant $\widehat\lambda$ is called the contraction rate. 
\end{definition}

Note that when the metric $G$ is constant, the condition (\ref{contractionGAS1}) is known as the Demidovich's condition \cite{pavlov,pavlov1,demidovich}. It is shown in \cite{pavlov} that such condition implies $\delta$-GAS and the convergent system property for $\Sigma$.

\begin{definition}[\cite{majid1}]
\label{contraction1}
Let $\Sigma=(\mathbb{R}^n,\mathsf{U},\mathcal{U},f)$ be a smooth control system on $\R^n$ equipped with a Riemannian metric $G$. The metric $G$ is said to be a contraction metric, with respect to states and inputs, for system $\Sigma$ if there exist some $\widehat\lambda\in\mathbb{R}^+$ and $\alpha\in\mathbb{R}_0^+$ such that:
\begin{equation}
\label{contraction3ISS}
\langle{X},X\rangle_{F}+2\bigg\langle\frac{\partial{f}}{\partial{u}}Y,X\bigg\rangle_{G}\leq-\widehat\lambda\langle{X},X\rangle_{G}+\alpha{\langle{X},X\rangle^{\frac{1}{2}}_G}\langle{Y,Y}\rangle^{\frac{1}{2}}_{I_m},
\end{equation}
for $F(x,u)=\left(\frac{\partial{f}}{\partial{x}}\right)^TG(x)+G(x)\frac{\partial{f}}{\partial{x}}+\frac{\partial{G}}{\partial{x}}f(x,u)$, any \mbox{$X\in\mathbb{R}^n$}, $x\in\R^n$, $u\in\mathsf{U}$, and $Y\in\R^m$. The constant $\widehat\lambda$ is called the contraction rate. 
\end{definition}

The following theorem shows that existence of a contraction metric, with respect to states and inputs, (resp. with respect to states) implies $\delta_\exists$-ISS (resp. $\delta_\exists$-GAS).

\begin{theorem}
\label{theoremISS}
Let $\Sigma=(\mathbb{R}^n,\mathsf{U},\mathcal{U},f)$ be a smooth control system on $\R^n$ equipped with a Riemannian metric $G$ and let $\mathsf{U}$ be a convex set. If the metric $G$ is a contraction metric, with respect to states and inputs, (resp. with respect to states) for system $\Sigma$ and $\left(\R^n,\mathbf{d}_G\right)$ is a complete metric space, then $\Sigma$ is $\delta_\exists$-ISS (resp. $\delta_\exists$-GAS).
\end{theorem}

\begin{proof}
Since $\left(\R^n,\mathbf{d}_G\right)$ is a complete metric space, using the Hopf-Rinow theorem \cite{petersen}, we conclude that $\R^n$ with respect to the metric $G$ is geodesically complete. The rest of the proof is inspired by the proof of Theorem 2 in \cite{aghannan}. Consider two points $x$ and $x'$ in $\mathbb{R}^n$ and a geodesic $\chi:[0,1]\rightarrow{\mathbb{R}^n}$ joining $x=\chi(0)$ and $x'=\chi(1)$. The geodesic distance between the points $x$ and $x'$ is given by:
\begin{equation}
\mathbf{d}_G(x,x')=\int_0^1{\sqrt{\left(\frac{d\chi(s)}{ds}\right)^TG(\chi(s))\frac{d\chi(s)}{ds}}ds}. \label{geodesic}
\end{equation}
Consider the straight line $\widehat\chi_t(s)=(1-s)\upsilon(t)+s\upsilon'(t)$, for fixed $t\in\mathbb{R}_0^+$, fixed $\upsilon,\upsilon'\in \mathcal{U}$, and for any $s\in[0,1]$. The curve $\widehat{\chi}_t$ is a geodesic, with respect to the Euclidean metric, on the subset $\mathsf{U}\subseteq\mathbb{R}^m$ joining $\upsilon(t)=\widehat\chi_t(0)$ and $\upsilon'(t)=\widehat\chi_t(1)$. Consider also the input curve $\upsilon_s$ defined by $\upsilon_s(t)=\widehat{\chi}_t(s)$. Let ${l}(t)$ be the length of the curve $\xi_{\chi(s){\upsilon_s}}(t)$ parametrized by $s$ and with respect to the metric $G$, i.e.:
\begin{equation}
l(t)=\int_0^1\sqrt{\delta{\xi}^TG(\xi_{\chi(s){\upsilon_s}}(t))\delta{\xi}}ds,~~\text{with}~~\delta{\xi}=\frac{\partial}{\partial{s}}\xi_{\chi(s){\upsilon_s}}(t). \label{length}
\end{equation}
In the rest of the proof, we drop the argument of the metric $G$ for the sake of simplicity. By taking the derivative of (\ref{length}) with respect to time, we obtain:
\begin{align}\nonumber
\frac{d}{dt}{l}(t)=&\int_0^1\frac{\frac{d}{dt}\left(\delta{\xi}^TG\delta{\xi}\right)}{2\sqrt{\delta{\xi}^TG\delta{\xi}}}ds \notag \\
=&\int_0^1\frac{\delta{\xi}^T\left(\left(\frac{\partial{f}}{\partial{x}}\right)^TG+\frac{\partial{G}}{\partial{x}}f+G\frac{\partial{f}}{\partial{x}}\right)\delta{\xi}+2\delta{\upsilon}^T\left(\frac{\partial{f}}{\partial{u}}\right)^TG\delta{\xi}}{2\sqrt{\delta{\xi}^TG\delta{\xi}}}ds,~~\text{with}~~\delta{\upsilon}=\frac{\partial}{\partial{s}}\upsilon_s(t). \notag
\end{align}
Since $G$ is a contraction metric, with respect to states and inputs, with $\widehat\lambda$ and $\alpha$ the constants introduced in Definition \ref{contraction1}, the following inequality holds:
\begin{eqnarray}
\label{inequality3}
\frac{d}{dt}{l}(t)\leq-\frac{\widehat\lambda}{2}{l}(t)+\frac{\alpha}{2}\int_0^1\sqrt{\delta{\upsilon}^T\delta{\upsilon}}ds=-\frac{\widehat\lambda}{2}{l}(t)+\frac{\alpha}{2}\Vert\upsilon(t)-\upsilon'(t)\Vert.
\end{eqnarray}
Using (\ref{inequality3}) and the comparison principle \cite{khalil}, we obtain:
\begin{align}
{l}(t)\leq& \mathsf{e}^{-\frac{\widehat\lambda}{2}{t}}{l}(0)+\frac{\alpha}{2}\mathsf{e}^{-\frac{\widehat\lambda}{2}{t}}*\Vert\upsilon(t)-\upsilon'(t)\Vert \notag \\
\leq& \mathsf{e}^{-\frac{\widehat\lambda}{2}{t}}{l}(0)+\frac{\alpha}{{\widehat\lambda}}\left(1-\mathsf{e}^{-\frac{\widehat\lambda}{2}{t}}\right)\Vert\upsilon-\upsilon'\Vert_\infty\leq \mathsf{e}^{-\frac{\widehat\lambda}{2}{t}}{l}(0)+\frac{\alpha}{{\widehat\lambda}}\Vert\upsilon-\upsilon'\Vert_\infty, \notag
\end{align} 
where $*$ denotes the convolution integral\footnote{$\mathsf{e}^{-\frac{\widehat\lambda}{2}{t}}*\Vert\upsilon(t)-\upsilon'(t)\Vert=\int_0^t\mathsf{e}^{-\frac{\widehat\lambda}{2}{(t-\tau)}}\Vert\upsilon(\tau)-\upsilon'(\tau)\Vert d\tau$.}. From (\ref{geodesic}) and (\ref{length}), it can be seen that \mbox{$l(0)=\mathbf{d}_G(x,x')$}. However, for $t\in\mathbb{R}^+$, $l(t)$ is not necessarily the Riemannian distance function, determined by $G$, because $\xi_{\chi(s){\upsilon_s}}(t)$ is not necessarily a geodesic, implying that it is always bigger than or equal to the Riemannian distance function\footnote{Note that given a Riemannian metric $G$, the Riemannian distance function is the smallest distance, determined by $G$.}: $\mathbf{d}_G(\xi_{x\upsilon}(t),\xi_{x'{\upsilon'}}(t))\leq l(t)$, and, hence, the following inequality holds:
\begin{equation}\nonumber
\mathbf{d}_G\left(\xi_{x\upsilon}(t),\xi_{x'{\upsilon'}}(t)\right)\leq \mathsf{e}^{-\frac{\widehat\lambda}{2}{t}}\mathbf{d}_G(x,x')+\frac{\alpha}{{\widehat\lambda}}\Vert\upsilon-\upsilon'\Vert_{\infty},
\end{equation}
which, in turn, implies that $\Sigma$ is $\delta_\exists$-ISS. The proof for the case that $G$ is a contraction metric, with respect to states, can be readily verified by just enforcing $\delta\upsilon(t)=0$ and $\upsilon(t)=\upsilon'(t)$ for any $t\in\R_0^+$.
\end{proof}

Since completeness of the metric space $\left(\R^n,\mathbf{d}_G\right)$ is crucial to the previous proof, the following lemma provides a sufficient condition on the metric $G$ guaranteeing that $\left(\R^n,\mathbf{d}_G\right)$ is a complete metric space.

\begin{lemma}
\label{lemma20}
The Riemannian manifold $\R^n$ equipped with a Riemannian metric $G$, satisfying\footnote{This condition is nothing more than uniform positive definitness of $G$.} $\underline{\omega}\Vert{y}\Vert^2\leq y^TG(x)y$ for any $x,y\in\R^n$ and for some positive constant $\underline{\omega}$, is complete as a metric space, with respect to $\mathbf{d}_G$.
\end{lemma}
\begin{proof}
The proof was suggested to us by C. Manolescu. First, for each pair of points $x,y\in\R^n$ we define the path space:
\begin{equation}
\nonumber
\Omega(x,y)=\{\chi:[0,1]\rightarrow\R^n\,\,\vert\,\,\chi~\text{is piecewise smooth},~\chi(0)=x,~\text{and}~\chi(1)=y\}.
\end{equation}
Recall that a function $\chi:[a,b]\rightarrow\R^n$ is piecewise smooth if $\chi$ is continuous and if there exists a partitioning \mbox{$a=a_1<a_2<\cdots<a_k=b$} of $[a,b]$ such that $\chi|_{(a_i,a_{i+1})}$ is smooth for $i=1,\cdots,k-1$. We can then define the Riemannian distance function $\mathbf{d}_G(x,y)$ between points $x,y\in\R^n$ as
\begin{equation}
\nonumber
\mathbf{d}_G(x,y)=\inf_{\chi\in\Omega(x,y)}\int_0^1{\sqrt{\left(\frac{d\chi(s)}{ds}\right)^TG(\chi(s))\frac{d\chi(s)}{ds}}ds}.
\end{equation}
It follows immediately that $\mathbf{d}_G$ is a metric on $\R^n$. The Riemannian manifold $\R^n$ is a complete metric space, equipped with the metric $\mathbf{d}_G$, if every Cauchy sequence\footnote{A sequence $\{x_n\}_{n=1}^\infty$ in a metric space $X$, equipped with a metric $\mathbf{d}$, is a Cauchy sequence if $\lim_{n,m\rightarrow\infty} \mathbf{d}(x_n,x_m)=0$.} of points in $\R^n$ has a limit in $\R^n$. Assume $\{x_n\}_{n=1}^\infty$ is a Cauchy sequence in $\R^n$, equipped with the metric $\mathbf{d}_G$. By using the assumption on $G$, we have
\begin{align}
\label{cauchy}
\mathbf{d}_G(x_n,x_m)=&\inf_{\chi\in\Omega(x_n,x_m)}\int_0^1{\sqrt{\left(\frac{d\chi(s)}{ds}\right)^TG(\chi(s))\frac{d\chi(s)}{ds}}ds}\\ \notag&\geq\sqrt{\underline\omega}\inf_{\chi\in\Omega(x_n,x_m)}\int_0^1{\sqrt{\left(\frac{d\chi(s)}{ds}\right)^T\frac{d\chi(s)}{ds}}ds}=\sqrt{\underline\omega}\Vert{x_n}-x_m\Vert.
\end{align} 
It is readily seen from the inequality (\ref{cauchy}) that the sequence $\{x_n\}_{n=1}^\infty$ is also a Cauchy sequence in $\R^n$ with respect to the Euclidean metric. Since the Riemannian manifold $\R^n$ with respect to the Euclidean metric is a complete metric space, the sequence $\{x_n\}_{n=1}^\infty$ converges to a point, named $x^*$, in $\R^n$. By picking a convex compact subset $D\subset\R^n$, containing $x^*$, and using Lemma 8.18 in \cite{lee}, we have: \mbox{$\overline{\omega}\Vert{y}\Vert^2\geq y^TG(x)y$} for any $y\in\R^n$, $x\in{D}$, and some positive constant $\overline{\omega}$. Since the sequence $\{x_n\}_{n=1}^\infty$ converges to $x^*\in D$, there exists some integer $N$ such that the sequence $\{x_n\}_{n=N}^\infty$ remains forever inside $D$. Hence, we have:
\begin{equation}\nonumber
\sqrt{\underline\omega}\Vert{x_n}-x^*\Vert\leq \mathbf{d}_G(x_n,x^*)\leq\sqrt{\overline\omega}\Vert{x_n}-x^*\Vert,
\end{equation}
for any $n>N$. Therefore, the sequence $\{x_n\}_{n=1}^\infty$ converges to $x^*\in\R^n$, equipped with the metric $\mathbf{d}_G$. Therefore, $\R^n$ with respect to the metric $\mathbf{d}_G$ is a complete metric space.

\end{proof}

Resuming, in this section we have provided a characterization of $\delta_{\exists}$-GAS and $\delta_{\exists}$-ISS in terms of the existence of $\delta_{\exists}$-GAS and $\delta_{\exists}$-ISS Lyapunov functions and we have provided sufficient conditions for $\delta_{\exists}$-GAS and $\delta_{\exists}$-ISS in terms of the existence of a contraction metric. Based on these results, in the next section, we propose a backstepping controller design procedure, providing controllers rendering control systems incrementally stable. Additionally, we will provide incremental Lyapunov functions and contraction metrics (the latter for smooth control systems).

\section{Backstepping Design Procedure}\label{back}
The backsteping method proposed here is inspired by the backstepping method, described in \cite{pavlov}. Here, we will extend this approach to render the closed-loop system $\delta_\exists$-ISS and to construct $\delta_\exists$-ISS Lyapunov functions. Consider the following subclass of control systems:
\begin{equation}
    \Sigma:\left\{
                \begin{array}{ll}
                 \dot{\eta}=f(\eta,\zeta),\\
                 \dot{\zeta}=\upsilon,
                \end{array}
                \right.
\label{system1}
\end{equation}
where $x=\left[y^T,z^T\right]\in\R^{n_\eta+n_\zeta}$ is the state of $\Sigma$, $y$ and $z$ are initial conditions for $\eta$, $\zeta$-subsystems, respectively, and $\upsilon$ is the control input. 

In support of the main result of this section (Theorem \ref{theorem1}), we need the following technical result.

\begin{lemma}\label{lemma1}
Consider the following interconnected control system
\begin{equation}
    \Sigma:\left\{
                \begin{array}{ll}
                 \dot{\eta}=f(\eta,\zeta,\upsilon),\\
                 \dot{\zeta}=g(\zeta,\upsilon).
                \end{array}
                \right.
\label{system2}
\end{equation}
Let the $\eta$-subsystem be $\delta_\exists$-ISS with respect to $\zeta$, $\upsilon$ and let the $\zeta$-subsystem be $\delta_\exists$-ISS with respect to $\upsilon$ for some metrics $\mathbf{d}_\eta$ and $\mathbf{d}_\zeta$, respectively such that the solutions $\eta_{y\zeta\upsilon}$\footnote{Notation $\eta_{y\zeta\upsilon}$ denotes a trajectory of $\eta$-subsystem under the inputs $\zeta$ and $\upsilon$ from initial condition $y\in\R^{n_\eta}$.} and $\zeta_{z\upsilon}$ satisfy the following inequalities:
\begin{align}
\mathbf{d}_\eta\left(\eta_{y\zeta\upsilon}(t),\eta_{y'\zeta'\upsilon'}(t)\right)&\leq\beta_\eta\left(\mathbf{d}_\eta\left(y,y'\right),t\right)+\gamma_\zeta\left(\mathbf{d}_\zeta(\zeta,\zeta')_\infty\right)+\gamma_\upsilon\left(\left\Vert{\upsilon}-{\upsilon}'\right\Vert_{\infty}\right),\label{y-sub}\\
\mathbf{d}_\zeta\left(\zeta_{z\upsilon}(t),\zeta_{z'\upsilon'}(t)\right)&\leq\beta_\zeta\left(\mathbf{d}_\zeta\left(z,z'\right),t\right)+\widetilde\gamma_\upsilon\left(\left\Vert{\upsilon}-{\upsilon}'\right\Vert_{\infty}\right),\label{z-sub}
\end{align}
where $y,y'$ and $z,z'$ are the initial conditions for the $\eta$, $\zeta$-subsystems, respectively. Then, the interconnected control system $\Sigma$ in (\ref{system2}) is $\delta_\exists$-ISS with respect to $\upsilon$. 
\end{lemma}

\begin{proof}
The proof was inspired by the proof of Proposition 4.7 in \cite{angeli}. The essential differences lie in the choice of the metric for the overall system $\Sigma$ using the metrics for $\eta$, $\zeta$-subsystems. We provide the proof so that it can be easily compared with the proof in \cite{angeli}. Using (\ref{y-sub}), (\ref{z-sub}) and triangular inequality, the following chain of inequalities hold:
\begin{align}
\nonumber
\mathbf{d}_\eta\left(\eta_{y\zeta\upsilon}(t),\eta_{y'\zeta'\upsilon'}(t)\right)&\leq\beta_\eta\left(\mathbf{d}_\eta\left(\eta_{y\zeta\upsilon}(t/2),\eta_{y'\zeta'\upsilon'}(t/2)\right),t/2\right)
+\gamma_\zeta\left(\mathbf{d}_\zeta\left(\zeta,\zeta'\right)_{[t/2,\infty)}\right)+\gamma_\upsilon\left(\left\Vert{\upsilon}-{\upsilon}'\right\Vert_{[t/2,\infty)}\right)\\\notag
&\leq\beta_\eta\left(\beta_\eta\left(\mathbf{d}_\eta(y,y'),t/2\right)+\gamma_\zeta\left(\mathbf{d}_\zeta\left(\zeta,\zeta'\right)_\infty\right)+\gamma_\upsilon\left(\left\Vert{\upsilon}-{\upsilon}'\right\Vert_{\infty}\right),t/2\right)\\\notag
&~+\gamma_\zeta\left(\mathbf{d}_\zeta\left(\zeta,\zeta'\right)_{[t/2,\infty)}\right)+\gamma_\upsilon\left(\left\Vert{\upsilon}-{\upsilon}'\right\Vert_{[t/2,\infty)}\right)\\\notag
&\leq\beta_\eta\left(3\beta_\eta(\mathbf{d}_\eta(y,y'),t/2),t/2\right)+\beta_\eta\left(3\gamma_\zeta\left(\mathbf{d}_\zeta\left(\zeta,\zeta'\right)_\infty\right),t/2\right)+\beta_\eta\left(3\gamma_\upsilon\left(\Vert\upsilon-\upsilon'\Vert_\infty\right),0\right)\\\notag
&~+\gamma_\zeta\left(\mathbf{d}_\zeta\left(\zeta,\zeta'\right)_{[t/2,\infty)}\right)+\gamma_\upsilon\left(\left\Vert{\upsilon}-{\upsilon}'\right\Vert_{[t/2,\infty)}\right)\\\notag
&\leq\beta_\eta\left(3\beta_\eta(\mathbf{d}_\eta(y,y'),t/2),t/2\right)+\beta_\eta\left(3\gamma_\zeta\left(\beta_\zeta\left(\mathbf{d}_\zeta\left(z,z'\right),0\right)+\widetilde\gamma_\upsilon\left(\left\Vert{\upsilon}-{\upsilon}'\right\Vert_{\infty}\right)\right),t/2\right)\\\notag
&~+\beta_\eta\left(3\gamma_\upsilon\left(\Vert\upsilon-\upsilon'\Vert_\infty\right),0\right)+\gamma_\zeta\left(\mathbf{d}_\zeta\left(\zeta,\zeta'\right)_{[t/2,\infty)}\right)+\gamma_\upsilon\left(\left\Vert{\upsilon}-{\upsilon}'\right\Vert_{[t/2,\infty)}\right)\\\notag
&\leq\beta_\eta\left(3\beta_\eta(\mathbf{d}_\eta(y,y'),t/2),t/2\right)+\beta_\eta\left(3\gamma_\zeta\left(2\beta_\zeta\left(\mathbf{d}_\zeta\left(z,z'\right),0\right)\right),t/2\right)\\\notag
&~+\beta_\eta\left(3\gamma_\zeta\left(2\widetilde\gamma_\upsilon\left(\left\Vert{\upsilon}-{\upsilon}'\right\Vert_{\infty}\right)\right),t/2\right)+\beta_\eta\left(3\gamma_\upsilon\left(\Vert\upsilon-\upsilon'\Vert_\infty\right),0\right)\\\notag
&~+\gamma_\zeta\left(\mathbf{d}_\zeta\left(\zeta,\zeta'\right)_{[t/2,\infty)}\right)+\gamma_\upsilon\left(\left\Vert{\upsilon}-{\upsilon}'\right\Vert_{[t/2,\infty)}\right)\\\notag
&\leq\beta_\eta\left(3\beta_\eta(\mathbf{d}_\eta(y,y'),t/2),t/2\right)+\beta_\eta\left(3\gamma_\zeta\left(2\beta_\zeta\left(\mathbf{d}_\zeta\left(z,z'\right),0\right)\right),t/2\right)\\\notag
&~+\beta_\eta\left(3\gamma_\zeta\left(2\widetilde\gamma_\upsilon\left(\left\Vert{\upsilon}-{\upsilon}'\right\Vert_{\infty}\right)\right),0\right)+\beta_\eta\left(3\gamma_\upsilon\left(\Vert\upsilon-\upsilon'\Vert_\infty\right),0\right)\\\notag
&~+\gamma_\zeta\left(\beta_\zeta\left(\mathbf{d}_\zeta\left(z,z'\right),t/2\right)+\widetilde\gamma_\upsilon\left(\Vert\upsilon-\upsilon'\Vert_\infty\right)\right)+\gamma_\upsilon\left(\left\Vert{\upsilon}-{\upsilon}'\right\Vert_{[t/2,\infty)}\right)\\\notag
&\leq\beta_\eta\left(3\beta_\eta(\mathbf{d}_\eta(y,y'),t/2),t/2\right)+\beta_\eta\left(3\gamma_\zeta\left(2\beta_\zeta\left(\mathbf{d}_\zeta\left(z,z'\right),0\right)\right),t/2\right)\\\notag
&~+\beta_\eta\left(3\gamma_\zeta\left(2\widetilde\gamma_\upsilon\left(\left\Vert{\upsilon}-{\upsilon}'\right\Vert_{\infty}\right)\right),0\right)+\beta_\eta\left(3\gamma_\upsilon\left(\Vert\upsilon-\upsilon'\Vert_\infty\right),0\right)\\\notag
&~+\gamma_\zeta\left(2\beta_\zeta\left(\mathbf{d}_\zeta\left(z,z'\right),t/2\right)\right)+\gamma_\zeta\left(2\widetilde\gamma_\upsilon\left(\Vert\upsilon-\upsilon'\Vert_\infty\right)\right)+\gamma_\upsilon\left(\left\Vert{\upsilon}-{\upsilon}'\right\Vert_{\infty}\right)\\\label{y-sub-grow}
&\leq\widehat\beta\left(\mathbf{d}_\eta(y,y'),t\right)+\widetilde\beta\left(\mathbf{d}_\zeta(z,z'),t\right)+\widehat\gamma\left(\Vert\upsilon-\upsilon'\Vert_\infty\right),
\end{align}
where $\widehat\gamma\in\mathcal{K}_\infty$ and $\widehat\beta,\widetilde\beta\in\mathcal{KL}$ are defined as following:
\begin{align}
\nonumber
\widehat\gamma(r)&=\beta_\eta\left(3\gamma_\zeta\left(2\widetilde\gamma_\upsilon(r)\right),0\right)+\beta_\eta\left(3\gamma_\upsilon(r),0\right)+\gamma_\zeta\left(2\widetilde\gamma_\upsilon(r)\right)+\gamma_\upsilon(r),\\\notag
\widehat\beta(r,t)&=\beta_\eta\left(3\beta_\eta\left(r,t/2\right),t/2\right),\\\notag
\widetilde\beta(r,t)&=\beta_\eta\left(3\gamma_\zeta\left(2\beta_\zeta(r,0)\right),t/2\right)+\gamma_\zeta\left(2\beta_\zeta(r,t/2)\right).
\end{align} 
Now we define a new metric $\mathbf{d}:\R^{n_\eta+n_\zeta}\times\R^{n_\eta+n_\zeta}\rightarrow\R^+_0$ by: 
\begin{equation}\nonumber
\mathbf{d}(x,x')=\mathbf{d}_\eta(y,y')+\mathbf{d}_\zeta(z,z'),
\end{equation}
for any $x=\left[{y}^T,{z}^T\right]^T\in\R^{n_\eta+n_\zeta}$ and $x'=\left[{y'}^T,{z'}^T\right]^T\in\R^{n_\eta+n_\zeta}$. It can be readily checked that $\mathbf{d}$ satisfies all three conditions of a metric. By defining $\xi_{x\upsilon}=\left[\eta_{y\zeta\upsilon}^T,\zeta_{z\upsilon}^T\right]^T$, using inequalities (\ref{z-sub}) and (\ref{y-sub-grow}), and for any $t\in\R_0^+$, any $x,x'\in\R^{n_\eta+n_\zeta}$, and any $\upsilon,\upsilon'\in\mathcal{U}$, we obtain:
\begin{eqnarray}\notag
\mathbf{d}\left(\xi_{x\upsilon}(t),\xi_{x'\upsilon'}(t)\right)&=&\mathbf{d}_\eta\left(\eta_{y\zeta\upsilon}(t),\eta_{y'\zeta'\upsilon'}(t)\right)+\mathbf{d}_\zeta\left(\zeta_{z\upsilon}(t),\zeta_{z'\upsilon'}(t)\right)\\\notag
&\leq&\widehat\beta\left(\mathbf{d}_\eta(y,y'),t\right)+\widetilde\beta\left(\mathbf{d}_\zeta(z,z'),t\right)+\widehat\gamma\left(\Vert\upsilon-\upsilon'\Vert_\infty\right)\\\notag
&&+\beta_\zeta\left(\mathbf{d}_\zeta\left(z,z'\right),t\right)+\widetilde\gamma_\upsilon\left(\left\Vert{\upsilon}-{\upsilon}'\right\Vert_{\infty}\right)\\\notag
&\leq&\widehat\beta\left(\mathbf{d}_\eta(y,y')+\mathbf{d}_\zeta(z,z'),t\right)+\widetilde\beta\left(\mathbf{d}_\eta(y,y')+\mathbf{d}_\zeta(z,z'),t\right)\\\notag
&&+\beta_\zeta\left(\mathbf{d}_\eta(y,y')+\mathbf{d}_\zeta\left(z,z'\right),t\right)+\widehat\gamma\left(\Vert\upsilon-\upsilon'\Vert_\infty\right)+\widetilde\gamma_\upsilon\left(\left\Vert{\upsilon}-{\upsilon}'\right\Vert_{\infty}\right)\\\notag
&\leq&\beta\left(\mathbf{d}(x,x'),t\right)+\gamma\left(\Vert\upsilon-\upsilon'\Vert_\infty\right),
\end{eqnarray}
where $\beta\in\mathcal{KL}$ and $\gamma\in\mathcal{K}_\infty$ are defined as following:
\begin{eqnarray}
\nonumber
\beta(r,t)&=&\widehat\beta(r,t)+\widetilde\beta(r,t)+\beta_\zeta(r,t),\\\notag
\gamma(r)&=&\widehat\gamma(r)+\widetilde\gamma_\upsilon(r).
\end{eqnarray}
Hence, the overall system $\Sigma$ of the form (\ref{system2}) is $\delta_\exists$-ISS with respect to $\upsilon$. 

\end{proof}

We can now state the main result, on a backstepping controller design approach for the control system $\Sigma$ in (\ref{system1}), rendering the resulting closed-loop system $\delta_\exists$-ISS.
\begin{theorem}\label{theorem1}
Consider the control system $\Sigma$ of the form (\ref{system1}). Suppose there exists a continuously differentiable function $\psi:\R^{n_\eta}\rightarrow\R^{n_\zeta}$ such that the control system 
\begin{equation}\label{eta_sub}
\Sigma_\eta: \dot\eta=f(\eta,\psi(\eta)+\widetilde\upsilon)
\end{equation}
is $\delta_\exists$-ISS with respect to the input $\widetilde\upsilon$. Then for any $\lambda\in\R^+$, the state feedback control law: 
\begin{equation}\label{law}
\upsilon=k(\eta,\zeta,\widehat\upsilon)=-\lambda(\zeta-\psi(\eta))+\frac{\partial{\psi}}{\partial{y}}(\eta)f(\eta,\zeta)+\widehat\upsilon
\end{equation}
renders the control system $\Sigma$ $\delta_\exists$-ISS with respect to the input $\widehat\upsilon$.
\end{theorem}

\begin{proof}
Consider the following coordinate transformation:
\begin{eqnarray}
\chi=\left[{\begin{array}{c}\chi_1\\\chi_2\\\end{array}}\right]=\phi(\xi)=\left[{\begin{array}{c}\eta\\\zeta-\psi(\eta)\\\end{array}}\right], \label{trancor}
\end{eqnarray}
where $\xi=\left[\eta^T,\zeta^T\right]^T$. In the new coordinate $\chi$, we obtain the following dynamics:
\begin{equation}
    \widehat\Sigma:\left\{
                \begin{array}{ll}
                 \dot{\chi_1}=f\left(\chi_1,\psi(\chi_1)+\chi_2\right),\\
                 \dot{\chi_2}=\upsilon-\frac{\partial{\psi}}{\partial{y}}(\chi_1)f\left(\chi_1,\psi(\chi_1)+\chi_2\right).
                \end{array}
                \right.
\label{system3}
\end{equation}
The proposed control law (\ref{law}), given in the new coordinate $\chi$ by
\begin{equation}\label{law1}
\upsilon=k(\chi_1,\chi_2+\psi(\chi_1),\widehat\upsilon)=-\lambda\chi_2+\frac{\partial{\psi}}{\partial{y}}(\chi_1)f\left(\chi_1,\psi(\chi_1)+\chi_2\right)+\widehat\upsilon,
\end{equation}
transforms the control system $\widehat\Sigma$ into:
\begin{equation}
    \widetilde\Sigma:\left\{
                \begin{array}{ll}
                 \dot{\chi_1}=f(\chi_1,\psi(\chi_1)+\chi_2),\\
                 \dot{\chi_2}=-\lambda\chi_2+\widehat\upsilon.
                \end{array}
                \right.
\label{system4}
\end{equation}
Due to the choice of $\psi$, the $\chi_1$-subsystem of $\widetilde\Sigma$ is $\delta_\exists$-ISS with respect to $\chi_2$. It can be easily verified that the $\chi_2$-subsystem is input-to-state stable with respect to the input $\widehat\upsilon$. Since any ISS linear control system is also $\delta$-ISS \cite{angeli}, $\chi_2$-subsystem is also $\delta$-ISS\footnote{We recall that $\delta$-ISS property is equivalentt to $\delta_\exists$-ISS property with respect to the Euclidean metric.} with respect to $\widehat\upsilon$. Therefore, using Lemma \ref{lemma1}, we conclude that the control system $\widetilde\Sigma$ is $\delta_\exists$-ISS with respect to the input $\widehat\upsilon$. Since, $\delta_\exists$-ISS property is coordinate invariant \cite{majid1}, we conclude that the original control system $\Sigma$ in (\ref{system1}) equipped with the state feedback control law in (\ref{law}) is $\delta_\exists$-ISS with respect to the input $\widehat\upsilon$ which completes the proof. 
\end{proof}

\begin{remark}
The $\delta_\exists$-ISS property of system $\Sigma_\eta$ in (\ref{eta_sub}) can be stablished, for example, using the approaches provided in \cite{pavlov,wouw} for some relevant classes of control systems (such as piece-wise affine systems and Lur'e-type systems). 
\end{remark}

\begin{remark}
The result of Theorem \ref{theorem1} can be extended to the case that we have arbitrary number of integrators:
\begin{equation}\notag
    \Sigma:\left\{
                \begin{array}{l}
                 \dot{\eta}=f(\eta,\zeta_1),\\
                 \dot{\zeta_1}=\zeta_2,\\
                 \hspace{5mm}\vdots\\
                 \dot{\zeta_k}=\upsilon.
                \end{array}
                \right.
\end{equation}
Note that in this case, the functions $f$ and $\psi$ must be sufficiently differentiable. 
\end{remark}


Although the proposed approach in Theorem \ref{theorem1} provides a controller rendering the control system $\Sigma$ of the form (\ref{system1}) $\delta_\exists$-ISS, it does not provide a way of constructing $\delta_\exists$-ISS Lyapunov functions or contraction metrics. In the next lemmas, we show how to construct incremental Lyapunov functions and contraction metrics for the resulting closed-loop system, recursively. Note that the constructed incremental Lyapunov functions can be used as a necessary tool in the analysis in \cite{girard2,girard3,julius,kapinski}. We will show in the example section how explicit availability of an incremental Lyapunov function helps us to use the results in \cite{girard2} to construct a finite bisimilar abstraction for an incrementally input-to-state stable (non-smooth) control system.

\begin{lemma}\label{lemma2}
Consider the control system $\Sigma$ of the form (\ref{system1}). Suppose there exists a continuously differentiable function $\psi:\R^{n_\eta}\rightarrow\R^{n_\zeta}$ such that the smooth function $\widehat V:\R^{n_\eta}\times\R^{n_\eta}\rightarrow\R_0^+$ is a $\delta_\exists$-ISS Lyapunov function for the control system 
\begin{equation}\label{eta_sub1}
\Sigma_\eta: \dot\eta=f(\eta,\psi(\eta)+\widetilde\upsilon),
\end{equation}
and with respect to the control input $\widetilde\upsilon$. Assume that $\widehat{V}$ satisfies condition (iii) in Definition \ref{delta_GAS_Lya} for some $\kappa\in\R^+$ and some $\sigma\in\mathcal{K}_\infty$, satisfying $\sigma(r)\leq\widehat\kappa{r^2}$ for some $\widehat\kappa\in\R^+$ and any $r\in\R_0^+$. Then the function $\widetilde{V}:\R^{n_\eta+n_\zeta}\times\R^{n_\eta+n_\zeta}\rightarrow\R^+_0$, defined as $$\widetilde V(x,x')=\widehat{V}(y,y')+\Vert(z-\psi(y))-(z'-\psi(y'))\Vert^2,$$ where $x=\left[y^T,z^T\right]^T$ and $x'=\left[y'^T,z'^T\right]^T$, is a $\delta_\exists$-ISS Lyapunov function for $\Sigma$ as in (\ref{system1}) equipped with the state feedback control law (\ref{law}) for all $\lambda\geq\frac{\kappa+\widehat\kappa+1}{2}$.
\end{lemma}

\begin{proof}
As explained in the proof of Theorem \ref{theorem1}, using the proposed state feedback control law (\ref{law}) and the coordinate transformation $\phi$ in (\ref{trancor}), the control system $\Sigma$ of the form (\ref{system1}) is transformed to the control system $\widetilde\Sigma$ in (\ref{system4}). Now we show that $$V(\widehat{x},\widehat{x}')=\widehat{V}(\widehat{x}_1,\widehat{x}'_1)+(\widehat{x}_2-\widehat{x}'_2)^T(\widehat{x}_2-\widehat{x}'_2),$$ is a $\delta_\exists$-ISS Lyapunov function for $\widetilde\Sigma$, where $\widehat{x}=\left[\widehat{x}_1^T,\widehat{x}_2^T\right]^T$ and $\widehat{x}'=\left[\widehat{x}'^T_1,\widehat{x}'^T_2\right]^T$ are the states of $\widetilde\Sigma$ and $\widehat{x}_1,\widehat{x}'_1$, and $\widehat{x}_2,\widehat{x}'_2$ are the states of $\chi_1$, $\chi_2$-subsystems, respectively. Since $\widehat{V}$ is a $\delta_\exists$-ISS Lyapunov function for $\chi_1$-subsystem when $\chi_2$ is the input, it satisfies condition (i) in Definition \ref{delta_GAS_Lya} using a metric $\mathbf{d}$ as follows:$$\underline\alpha(\mathbf{d}(\widehat{x}_1,\widehat{x}'_1))\leq \widehat V(\widehat{x}_1,\widehat{x}'_1)\leq\overline\alpha(\mathbf{d}(\widehat{x}_1,\widehat{x}'_1)),$$ for some $\underline\alpha,\overline\alpha\in\mathcal{K}_\infty$. Now we define a new metric $\widehat{\mathbf{d}}:\R^{n_\eta+n_\zeta}\times\R^{n_\eta+n_\zeta}\rightarrow\R^+_0$ by 
\begin{equation}\label{new_metric}
\widehat{\mathbf{d}}(\widehat{x},\widehat{x}')=\mathbf{d}(\widehat{x}_1,\widehat{x}'_1)+\Vert{\widehat{x}_2}-\widehat{x}'_2\Vert. 
\end{equation}
It can be readily checked that $\widehat{\mathbf{d}}$ satisfies all three conditions of a metric. Using metric $\widehat{\mathbf{d}}$, function $V$ satisfies condition (i) in Definition \ref{delta_GAS_Lya} as follows:$$\underline\mu\left(\widehat{\mathbf{d}}(\widehat{x},\widehat{x}')\right)\leq V(\widehat{x},\widehat{x}')\leq\overline\mu\left(\widehat{\mathbf{d}}(\widehat{x},\widehat{x}')\right),$$ where $\underline\mu,\overline\mu\in\mathcal{K}_\infty$, $\underline{\mu}\left(\widehat{\mathbf{d}}(\widehat{x},\widehat{x}')\right)=\underline\alpha(\mathbf{d}(\widehat{x}_1,\widehat{x}'_1))+\Vert \widehat{x}_2-\widehat{x}'_2\Vert^2$, and $\overline{\mu}\left(\widehat{\mathbf{d}}(\widehat{x},\widehat{x}')\right)=\overline\alpha(\mathbf{d}(\widehat{x}_1,\widehat{x}'_1))+\Vert \widehat{x}_2-\widehat{x}'_2\Vert^2$. Now we show that $V$ satisfies condition (iii) in Definition \ref{delta_GAS_Lya} for $\widetilde\Sigma$. Since $\widehat{V}$ is a $\delta_\exists$-ISS Lyapunov function for $\chi_1$-subsystem when $\chi_2$ is the input, $\lambda\geq\frac{\kappa+\widehat\kappa+1}{2}$, $\sigma(r)\leq\widehat\kappa{r^2}$, and using the Cauchy Schwarz inequality, we have:
\begin{align}\notag
\frac{\partial{V}}{\partial{\widehat{x}}}\left[f(\widehat{x}_1,\psi(\widehat{x}_1)+\widehat{x}_2)^T,-\lambda\widehat{x}_2^T+\widehat{u}^T\right]^T+\frac{\partial{V}}{\partial{\widehat{x}'}}\left[f(\widehat{x}'_1,\psi(\widehat{x}'_1)+\widehat{x}'_2)^T,-\lambda\widehat{x}'^T_2+\widehat{u}'^T\right]^T&\leq\\\notag
\frac{\partial{\widehat{V}}}{\partial{\widehat{x}_1}}f(\widehat{x}_1,\psi(\widehat{x}_1)+\widehat{x}_2)+\frac{\partial{\widehat{V}}}{\partial{\widehat{x}'_1}}f(\widehat{x}'_1,\psi(\widehat{x}'_1)+\widehat{x}'_2)+2(\widehat{x}_2-\widehat{x}'_2)^T\left(-\lambda\widehat{x}_2+\widehat{u}\right)-2(\widehat{x}_2-\widehat{x}'_2)^T\left(-\lambda\widehat{x}'_2+\widehat{u}'\right)&\leq\\\notag
-\kappa\widehat{V}(\widehat{x}_1,\widehat{x}'_1)+\sigma(\Vert\widehat{x}_2-\widehat{x}'_2\Vert)-2\lambda\Vert\widehat{x}_2-\widehat{x}'_2\Vert^2+2(\widehat{x}_2-\widehat{x}'_2)^T(\widehat{u}-\widehat{u}')&\leq\\\notag
-\kappa\widehat{V}(\widehat{x}_1,\widehat{x}'_1)+\widehat\kappa\Vert\widehat{x}_2-\widehat{x}'_2\Vert^2-2\lambda\Vert\widehat{x}_2-\widehat{x}'_2\Vert^2+2\Vert \widehat{x}_2-\widehat{x}'_2\Vert\Vert\widehat{u}-\widehat{u}'\Vert&\leq\\\notag
-\kappa\widehat{V}(\widehat{x}_1,\widehat{x}'_1)+\widehat\kappa\Vert\widehat{x}_2-\widehat{x}'_2\Vert^2-2\lambda\Vert\widehat{x}_2-\widehat{x}'_2\Vert^2+\Vert \widehat{x}_2-\widehat{x}'_2\Vert^2+\Vert\widehat{u}-\widehat{u}'\Vert^2&\leq\\\notag
-\kappa V(\widehat{x},\widehat{x}')+\Vert\widehat{u}-\widehat{u}'\Vert^2&.
\end{align}  
The latter inequality implies that $V$ is a $\delta_\exists$-ISS Lyapunov function for $\widetilde\Sigma$. Since $\delta_\exists$-ISS Lyapunov functions are coordinate-invariant \cite{majid4}, we conclude that the function $\widetilde{V}:\R^{n_\eta+n_\zeta}\times\R^{n_\eta+n_\zeta}\rightarrow\R^+_0$, defined by $$\widetilde V(x,x')=V(\phi(x),\phi(x'))=\widehat{V}(y,y')+\Vert(z-\psi(y))-(z'-\psi(y'))\Vert^2,$$ is a $\delta_\exists$-ISS Lyapunov function for $\Sigma$, as in (\ref{system1}) equipped with the state feedback control law in (\ref{law}).
\end{proof}

\begin{remark}
One can use the LMI based results in \cite{pavlov,pavlov2,wouw} to find a quadratic $\delta_\exists$-ISS Lyapunov function for system $\Sigma_\eta$ in (\ref{eta_sub1}). 
\end{remark}

\begin{remark}
It can be verified that the backstepping design approach for strict-feedback form control systems, proposed in \cite{majid4}, is a special case of the results in Lemma \ref{lemma2}. The results in \cite{majid4} can be easily obtained by recursively applying the results proposed in Lemma \ref{lemma2}. Moreover, one can construct a metric $\mathbf{d}$ for a strict-feedback form control system, satisfying (\ref{dISS}), by recursively applying the construction in (\ref{new_metric}) and applying the change of coordinate in (\ref{trancor}). 
\end{remark}

The next lemma shows how to construct contraction metrics for the closed-loop system resulting from the backstepping controller synthesis technique in Theorem \ref{theorem1}.

\begin{lemma}\label{lemma3}
Consider the control system $\Sigma$ of the form (\ref{system1}) and assume function $f$ is smooth. Suppose there exists a continuously differentiable function $\psi:\R^{n_\eta}\rightarrow\R^{n_\zeta}$ such that the metric $\widehat G:\R^{n_\eta}\rightarrow\R^{n_\eta\times n_\eta}$ is a contraction metric, with respect to states and inputs, for the control system 
\begin{equation}\notag
\Sigma_\eta: \dot\eta=f(\eta,\psi(\eta)+\widetilde\upsilon),
\end{equation}
satisfying the condition (\ref{contraction3ISS}) for some $\widehat\lambda\in\mathbb{R}^+$ and $\alpha\in\mathbb{R}_0^+$. Then $$\widetilde{G}(x)=\left[{\begin{array}{cc}\widehat G(y)+\left(\frac{\partial\psi}{\partial{y}}\right)^T\frac{\partial\psi}{\partial{y}}&-\left(\frac{\partial\psi}{\partial{y}}\right)^T\\-\frac{\partial\psi}{\partial{y}}&I_{n_\zeta}\\\end{array}}\right],$$ where $x=\left[y^T,z^T\right]^T$, is a contraction metric, with respect to states and inputs, for $\Sigma$ as in (\ref{system1}) equipped with the state feedback control law in (\ref{law}) for all $\lambda>\frac{\alpha^2}{8\widehat\lambda}$.
\end{lemma}

\begin{proof}
As explained in the proof of Theorem \ref{theorem1}, using the proposed control law (\ref{law}) and the coordinate transformation $\phi$ in (\ref{trancor}), the control system $\Sigma$ of the form (\ref{system1}) is transformed to the control system $\widetilde\Sigma$ in (\ref{system4}). Now we show that the metric $$G(\widehat{x})=\left[{\begin{array}{cc}\widehat G(\widehat{x}_1)&0_{n_\eta\times n_\zeta}\\0_{n_\zeta\times n_\eta}&I_{n_\zeta}\\\end{array}}\right],$$ is a contraction metric, with respect to states and inputs, for $\widetilde\Sigma$, where $\widehat{x}=\left[\widehat{x}_1^T,\widehat{x}_2^T\right]^T$ is the state of $\widetilde\Sigma$, and $\widehat{x}_1$, and $\widehat{x}_2$ are states of $\chi_1,\chi_2$-subsystems, respectively. It can be easily seen that $G$ is positive definite because $\widehat{G}$ is positive definite since it is a contraction metric for $\Sigma_\eta$. Now we show that $G$ satisfies the condition (\ref{contraction3ISS}) for the control system $\widetilde\Sigma$. Since $\widehat{G}$ is a contraction metric, with respect to states and inputs, for $\chi_1$-subsystem when $\chi_2$ is the input, we have:
\begin{eqnarray}
\nonumber
&\widehat{X}_1^T\left(\left(\frac{\partial{f}}{\partial{\widehat{x}_1}}\right)^T\widehat G(\widehat{x}_1)+\widehat{G}(\widehat{x}_1)\frac{\partial{f}}{\partial{\widehat{x}_1}}+\frac{\partial{\widehat{G}}}{\partial{\widehat{x}_1}}f(\widehat{x}_1,\psi(\widehat{x}_1)+\widehat{x}_2)\right)\widehat{X}_1+2\widehat{X}_2^T\left(\frac{\partial{f}}{\partial{\widehat{x}_2}}\right)^T\widehat{G}(\widehat{x}_1)\widehat{X}_1\\
&\leq-\widehat\lambda\widehat{X}_1^T\widehat{G}(\widehat{x}_1)\widehat{X}_1+\alpha\left(\widehat{X}_1^T\widehat{G}(\widehat{x}_1)\widehat{X}_1\right)^\frac{1}{2}\left(\widehat{X}_2^T\widehat{X}_2\right)^\frac{1}{2}, \label{delta_ISS_cont}
\end{eqnarray}
for any $\widehat{X}_1,\widehat{x}_1\in\R^{n_\eta}$, $\widehat{X}_2,\widehat{x}_2\in\R^{n_\zeta}$, some $\widehat\lambda\in\R^+$, and some $\alpha\in\R_0^+$. By choosing $\lambda>\frac{\alpha^2}{8\widehat\lambda}$, using (\ref{delta_ISS_cont}), and the Cauchy Schwarz inequality we obtain:
\begin{align}
\nonumber
\left[\widehat{X}_1^T~\widehat{X}_2^T\right]\left(\left(\frac{\partial\left[f,-\lambda\widehat{x}_2+\widehat{u}\right]^T}{\partial{\widehat{x}}}\right)^TG(\widehat{x})+G(\widehat{x})\frac{\partial\left[f,-\lambda\widehat{x}_2+\widehat{u}\right]^T}{\partial{\widehat{x}}}+\frac{\partial{G}}{\partial{\widehat{x}}}\left[{\begin{array}{c}f(\widehat{x}_1,\psi(\widehat{x}_1)+\widehat{x}_2)\\-\lambda\widehat{x}_2+\widehat{u}\\\end{array}}\right]\right)\left[{\begin{array}{c}\widehat{X}_1\\\widehat{X}_2\\\end{array}}\right]+&\\\notag2Y^T\left[{\begin{array}{c}0_{n_\eta\times n_\zeta}\\I_{n_\zeta}\\\end{array}}\right]^TG(\widehat{x})\left[{\begin{array}{c}\widehat{X}_1\\\widehat{X}_2\\\end{array}}\right]&=\\\notag\left[\widehat{X}_1^T~\widehat{X}_2^T\right]\left[{\begin{array}{cc}\left(\frac{\partial{f}}{\partial{\widehat{x}_1}}\right)^T\widehat{G}(\widehat{x}_1)+\widehat{G}(\widehat{x}_1)\frac{\partial{f}}{\partial{\widehat{x}_1}}+\frac{\partial{\widehat{G}}}{\partial{\widehat{x}_1}}f&\widehat{G}(\widehat{x}_1)\frac{\partial{f}}{\partial{\widehat{x}_2}}\\\left(\frac{\partial{f}}{\partial{\widehat{x}_2}}\right)^T\widehat{G}(\widehat{x}_1)&-2\lambda I_{n_\zeta}\\\end{array}}\right]\left[{\begin{array}{c}\widehat{X}_1\\\widehat{X}_2\\\end{array}}\right]+2Y^T\widehat{X}_2&\leq\\\notag-\widehat\lambda\left\langle \widehat{X}_1,\widehat{X}_1\right\rangle_{\widehat{G}}+\alpha\left\langle\widehat{X}_1,\widehat{X}_1\right\rangle_{\widehat{G}}^{\frac{1}{2}}\left\langle\widehat{X}_2,\widehat{X}_2\right\rangle_{I_{n_\zeta}}^{\frac{1}{2}}-2\lambda\widehat{X}_2^T\widehat{X}_2+2Y^T\widehat{X}_2&\leq\\\notag
-\widetilde\lambda\left\langle\widehat{X}_1,\widehat{X}_1\right\rangle_{\widehat{G}}-\widetilde\lambda\widehat{X}_2^T\widehat{X}_2+2\sqrt{Y^TY}\sqrt{\widehat{X}_2^T\widehat{X}_2+\left\langle\widehat{X}_1,\widehat{X}_1\right\rangle_{\widehat{G}}}&\leq\\\notag
-\widetilde\lambda\left\langle\widehat{X},\widehat{X}\right\rangle_G+2\left\langle\widehat{X},\widehat{X}\right\rangle_{G}^{\frac{1}{2}}\left\langle Y,Y\right\rangle_{I_{n_\zeta}}^{\frac{1}{2}},
\end{align} 
for any $\widehat{X}=\left[\widehat{X}_1^T~\widehat{X}_2^T\right]^T\in\R^{n_\eta+n_\zeta}$, any $\widehat{x}=\left[\widehat{x}_1^T~\widehat{x}_2^T\right]^T\in\R^{n_\eta+n_\zeta}$, any $Y\in\R^{n_\zeta}$, and some $\widetilde\lambda\in\R^+$. Hence, $G$ is a contraction metric, with respect to states and inputs, for $\widetilde\Sigma$. Since a contraction metric, with respect to states and inputs, is coordinate invariant \cite{majid1}, we conclude that $\widetilde{G}=\phi^*G$ is a contraction metric, with respect to states and inputs, for $\Sigma$ as in (\ref{system1}) equipped with the state feedback control law in (\ref{law}). This completes the proof.
\end{proof}

\begin{remark}
It can be verified that the backstepping design approach for strict-feedback form control systems, proposed in \cite{majid1}, is a special case of the results in Lemma \ref{lemma3}. The results in \cite{majid1} can be easily obtained by recursively applying the results proposed in Lemma \ref{lemma3}. 
\end{remark}

\section{Example}
We refer the interested readers to the provided example in \cite{majid5}, illustrating the results in this paper on a four-dimensional non-smooth control system. Here, we study another non-smooth control system and use the results in this paper to explicitly construct a $\delta_\exists$-ISS Lyapunov function, which, in turn, is employed to construct a finite equivalent abstraction using the results in \cite{girard2}. Consider the following non-smooth control system:
\begin{equation}
    \Sigma:\left\{
                \begin{array}{l}
                 \dot{\eta_1}=\textsf{sat}(\eta_1)+\eta_1+5\zeta_1,\\
                 \dot{\zeta_1}=\zeta_1^2+\eta_1^2+\upsilon,
                \end{array}
                \right.
\label{example}
\end{equation}
where $\textsf{sat}:\R\rightarrow\R$ is the saturation function, defined by: 
\begin{equation}\notag
    \textsf{sat}(x)=\left\{
                \begin{array}{cl}
                 -1&\text{if}~x<-1,\\
                 x&\text{if}~\vert{x}\vert\leq1,\\
                 1&\text{if}~x>1.
                \end{array}
                \right.
\end{equation}

It can be readily verified that $\Sigma$ is unstable at $(0,0)$, implying that $\Sigma$ is not $\delta_\exists$-ISS. The results in \cite{jouffroy1,sharma,sharma1,majid1,majid4} can not be applied to design controllers that render the system $\Sigma$ $\delta_\exists$-ISS because the right hand side of $\Sigma$ in (\ref{example}) is not continuously differentiable. The results in \cite{pavlov} also can not be applied because these results would result in a closed-loop system which is input-to-state convergent rather than $\delta_\exists$-ISS. Note that here we require $\delta_\exists$-ISS and a $\delta_\exists$-ISS Lyapunov function in order to construct a finite equivalent abstraction using the results in \cite{girard2}. By introducing the feedback transformation $\widehat\upsilon=\zeta_1^2+\eta_1^2+\upsilon$, the control system $\Sigma$ is transformed into the following form:
\begin{equation}\notag
    \widehat\Sigma:\left\{
                \begin{array}{l}
                 \dot{\eta_1}=\textsf{sat}(\eta_1)+\eta_1+5\zeta_1,\\
                 \dot{\zeta_1}=\widehat\upsilon.
                \end{array}
                \right.
\end{equation}
Now by choosing $\psi(\eta_1)=-\eta_1$ and substituting $\psi(\eta_1)+\widetilde\upsilon$ instead of $\zeta_1$, we obtain the following $\eta$-subsystem:
\begin{equation}\notag
    \widehat\Sigma_\eta:\left\{
                \begin{array}{l}
                 \dot{\eta_1}=\textsf{sat}(\eta_1)+\eta_1+5(\psi(\eta_1)+\widetilde\upsilon)=\textsf{sat}(\eta_1)-4\eta_1+5\widetilde\upsilon.
                \end{array}
                \right.
\end{equation}
It remains to show that $\widehat\Sigma_\eta$ is $\delta_\exists$-ISS with respect to $\widetilde\upsilon$. By choosing the function $V_1(y_1,y_1')=(y_1-y_1')^2$, where $y_1$ and $y_1'$ are states of $\widehat\Sigma_\eta$, and using the Cauchy Schwarz inequality, we have that:
\begin{align}\notag
\frac{\partial{V_1}}{\partial{y_1}}\left(\textsf{sat}(y_1)-4y_1+5\widetilde{u}\right)+\frac{\partial{V_1}}{\partial{y'_1}}\left(\textsf{sat}\left(y'_1\right)-4y'_1+5\widetilde{u}'\right)&\leq\\\notag
-8(y_1-y_1')^2+2\vert{y_1}-y_1'\vert\vert\textsf{sat}(y_1)-\textsf{sat}(y'_1)\vert+10(y_1-y_1')(\widetilde{u}-\widetilde{u}')&\leq\\\notag-8(y_1-y_1')^2+2(y_1-y_1')^2+10(y_1-y_1')(\widetilde{u}-\widetilde{u}')&\leq\\\notag
-5(y_1-y_1')^2+25(\widetilde{u}-\widetilde{u}')^2,
\end{align}
showing that $V_1$ is a $\delta_\exists$-ISS Lyapunov function for $\widehat\Sigma_\eta$ and, hence, $\widehat\Sigma_\eta$ is $\delta_\exists$-ISS with respect to $\widetilde\upsilon$. By using the results in Theorem \ref{theorem1} for the control system $\widehat\Sigma$, we conclude that the state feedback control law:
\begin{align}\notag
\widehat\upsilon=k(\eta_1,\zeta_1,\bar\upsilon)=&-\lambda(\zeta_1-\psi(\eta_1))+\frac{\partial\psi}{\partial{y_1}}\dot\eta_1+\bar\upsilon\\\notag
=&-\lambda\left(\zeta_1+\eta_1\right)-\left(\textsf{sat}(\eta_1)+\eta_1+5\zeta_1\right)+\bar\upsilon,
\end{align}
makes the control system $\widehat\Sigma$ $\delta_\exists$-ISS with respect to input $\bar\upsilon$, for any $\lambda\in\R^+$. Therefore, the state feedback control law 
\begin{equation}\label{law5}
\upsilon=\widehat{k}(\eta_1,\zeta_1,\bar\upsilon)=k(\eta_1,\zeta_1,\bar\upsilon)-\eta_1^2-\zeta_1^2,
\end{equation}
makes the control system $\Sigma$ $\delta_\exists$-ISS with respect to input $\bar\upsilon$. 

Using Lemma \ref{lemma2}, we conclude that the function $V:\R^2\times\R^2\rightarrow\R_0^+$, defined by:
\begin{align}\notag
V(x,x')&=V_1(y_1,y_1')+\vert\left(z_1-\psi(y_1)\right)-\left(z_1'-\psi(y_1')\right)\vert^2\\\notag&=(y_1-y_1')^2+\left((z_1+y_1)-(z_1'+y_1')\right)^2\\\notag&=\left(x-x'\right)^TP\left(x-x'\right)=\left(x-x'\right)^T\left[ {\begin{array}{cc}
2&1\\
1&1\\
 \end{array} } \right]\left(x-x'\right),
\end{align}
where $x=\left[y_1,z_1\right]^T$ is the state of $\Sigma$, is a $\delta_\exists$-ISS Lyapunov function for the control system $\Sigma$ equipped with the state feedback control law $\widehat{k}$ in (\ref{law5}) with $\lambda>\frac{25+5+1}{2}$. Here, we choose $\lambda=16$.

It can be readily verified that the function $\widehat{V}(x,x')=\sqrt{V(x,x')}$ is also a $\delta_\exists$-ISS Lyapunov function for the control system $\Sigma$ equipped with the state feedback control law $\widehat{k}$ in (\ref{law5}) with $\lambda>\frac{25+5+1}{2}$, satisfying:
\begin{itemize}
\item[(i)] for any $x,x'\in\R^2$,\\
$\sqrt{\lambda_{\min}(P)}\left\Vert x-x'\right\Vert\leq \widehat{V}(x,x')\leq\sqrt{\lambda_{\max}(P)}\Vert x-x'\Vert;$
\item[(ii)] for any $x,x'\in\R^2$ and for any $\bar{u},\bar{u}'\in\overline{\mathsf{U}}\subseteq{\R}$,\\
$\frac{\partial{\widehat{V}}}{\partial{x}}f\left(x,\widehat{k}(x,\bar{u})\right)+\frac{\partial{\widehat{V}}}{\partial{x'}}f\left(x',\widehat{k}(x',\bar{u}')\right)\leq-2.5\widehat{V}(x,x')+\frac{\vert{\bar{u}}-\bar{u}'\vert}{\lambda_{\min}(P)};$
\item[(iii)] for any $x,y,z\in\R^2$\\
$\left\vert\widehat{V}(x,y)-\widehat{V}(x,z)\right\vert\leq\frac{\lambda_{\max}(P)}{\sqrt{\lambda_{\min}(P)}}\Vert y-z\Vert;$
\end{itemize}
where $\lambda_{\min}(P)$, and $\lambda_{\max}(P)$ stand for minimum and maximum eigenvalues of $P$. Note that the property (iii) is a consequence of Proposition 10.5 in \cite{paulo}.

Finite abstractions are simpler descriptions of control systems, typically with finitely many states, in which each state of the abstraction represents a collection or aggregate
of states in the control system. Similar finite abstractions are used in software and hardware modeling, which enables the composition of such abstractions with the finite abstraction of the control system. The result of this composition are finite abstractions capturing the behavior of the control system interacting with the digital computational devices. Once such abstractions are available, the methodologies and tools developed in computer science for verification and controller synthesis purposes can be easily employed to control systems, via these abstractions. However, for constructing a bisimilar finite abstraction, using the results in \cite{girard2} which does not impose any restriction on the sampling time, the control system is required to be incrementally stable and to exhibit an incremental Lyapunov function. The incremental stability property bounds the error propagations coming from discretization of the state space and input set in the process of constructing the finite bisimilar abstractions. We refer the interested readers to \cite{paulo} for more detailed information about the finite abstractions and their great advantages in controller synthesis problems.

Now, we construct a finite abstraction $S(\Sigma)$ for the control system $\Sigma$, equipped with the control input $\upsilon$ in (\ref{law5}), using the results in \cite{girard2}. We assume that $\bar\upsilon(t)\in\overline{\mathsf{U}}=[-10,~10]$, for any $t\in\R_0^+$, and $\bar\upsilon$ belongs to set $\overline{\mathcal{U}}$ that contains piecewise constant curves of duration $\tau=0.1$ second ($\tau$ is the sampling time) taking values in $\left[\overline{\mathsf{U}}\right]_{0.5}=\left\{\bar{u}\in\overline{\mathsf{U}}\,\,|\,\,\bar{u}=0.5k,~k\in\Ze\right\}$. We work on the subset $\mathsf{D}=[-1,~1]\times[-1,~1]$ of the state space $\Sigma$. For a given precision\footnote{The parameter $\varepsilon$ is the maximum error between a trajectory of the control system and its corresponding trajectory from the finite abstraction at times $k\tau$, $k\in\N_0$, with respect to the Euclidean metric.} $\varepsilon=0.1$ and using properties (i), (ii), and (iii) of $\widehat{V}$, we conclude that $\mathsf{D}$ should be quantized with resolution of $\eta=0.009$, using the results of Theorem 4.1 in \cite{girard2}. The state set of $S(\Sigma)$ is $[\mathsf{D}]_{\eta}=\left\{x\in\mathsf{D}\,\,|\,\,x_i=k_i\eta,~k_i\in\Ze,~i=1,2\right\}$. It can be readily seen that the set $[\mathsf{D}]_{\eta}$ is finite. The computation of the finite abstraction $S(\Sigma)$ was performed using the tool $\textsf{Pessoa}$ \cite{pessoa}. 
Using the computed finite abstraction, we can synthesize controllers, providing $\bar\upsilon$ in (\ref{law5}), satisfying specifications difficult to enforce with conventional controller design methods. Here, our objective is to design a controller navigating the trajectories of $\Sigma$, equipped with the control input $\upsilon$ in (\ref{law5}), to reach the target set $W=[-0.05,~0.05]\times[-0.05,~0.05]$, indicated with a red box in Figure \ref{fig1}, while avoiding the obstacles, indicated as blue boxes in Figure \ref{fig1}, and remain indefinitely inside $W$. Furthermore, we assume that the controller is implemented on a microprocessor, which is executing other tasks in addition to the control task. We consider a schedule with epochs of three time slots in which the first slot is allocated to the control task and the other two to other tasks. A time slot refers to a time interval of the form $[k\tau,(k+1)\tau[$ with $k\in\N$ and where $\tau$ is the sampling time. Therefore, the microprocessor schedules is given by (depending on the initial slot): 
$$
\mathsf{|auu|auu|auu|auu|auu|auu|auu|}\cdots,
$$
$$
\mathsf{|uua|uua|uua|uua|uua|uua|uua|}\cdots,
$$
$$
\mathsf{|uau|uau|uau|uau|uau|uau|uau|}\cdots,
$$
where $\mathsf{a}$ denotes a slot available for the control task and $\mathsf{u}$ denotes a slot allotted to other tasks. We assume that in unallocated time slots, the input $\bar\upsilon$ is identically zero. The schedulability constraint on the microprocessor can be represented by the finite system in Figure \ref{auto}.
\begin{figure}[t]
\begin{center}
\begin{tikzpicture}[shorten >=1pt,node distance=2cm,auto]
\tikzstyle{state}=[state with output]
\tikzstyle{every state}=[draw=brown!100,very thick,fill=brown!50,minimum size=0.8cm]
\node[state,initial,initial text=,initial where=above] (x1) {$q_{1}$ \nodepart{lower} {\textsf{a}}};
\node[state,initial,initial text=,initial where=above] (x2) [right of=x1] {$q_{2}$ \nodepart{lower} {\textsf{u}}};
\node[state,initial,initial text=,initial where=above] (x3) [right of=x2] {$q_{3}$ \nodepart{lower} {\textsf{u}}};
\path[->] (x1) edge node {} (x2)
(x2) edge node {} (x3)
(x3) edge [bend left] node {} (x1);
\end{tikzpicture}
\end{center}
\caption {Finite system describing the schedulability constraints. The lower part of the states are labeled with the outputs $\textsf{a}$ and $\textsf{u}$ denoting availability and unavailability of the microprocessor, respectively.}\label{auto}
\end{figure}
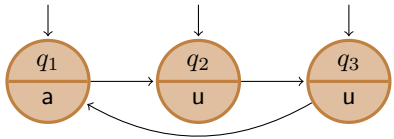

A controller, providing $\bar\upsilon$ in (\ref{law5}) and enforcing the specification has been designed by using standard algorithms from game theory, implemented in $\textsf{Pessoa}$, where the finite system is initialized from state $q_2$, see second sequence above. In Figure \ref{fig1}, we show the closed-loop trajectories of $\Sigma$, equipped with the control input $\upsilon$ in (\ref{law5}) (including the additional controller for $\bar\upsilon$) and stemming from the initial conditions $[0.8,~0.9]$ and $[-0.8,~-0.9]$. It is readily seen that the specifications are satisfied. In Figure \ref{fig2}, we show the evolution of the input signal $\bar\upsilon$ in (\ref{law5}) corresponding to the two initial conditions. It can be easily seen that the schedulability constraint is also satisfied, implying that the control input $\bar\upsilon$ is identically zero at unallocated time slots.

\begin{figure*}[t]
  \centering\includegraphics[width=15cm]{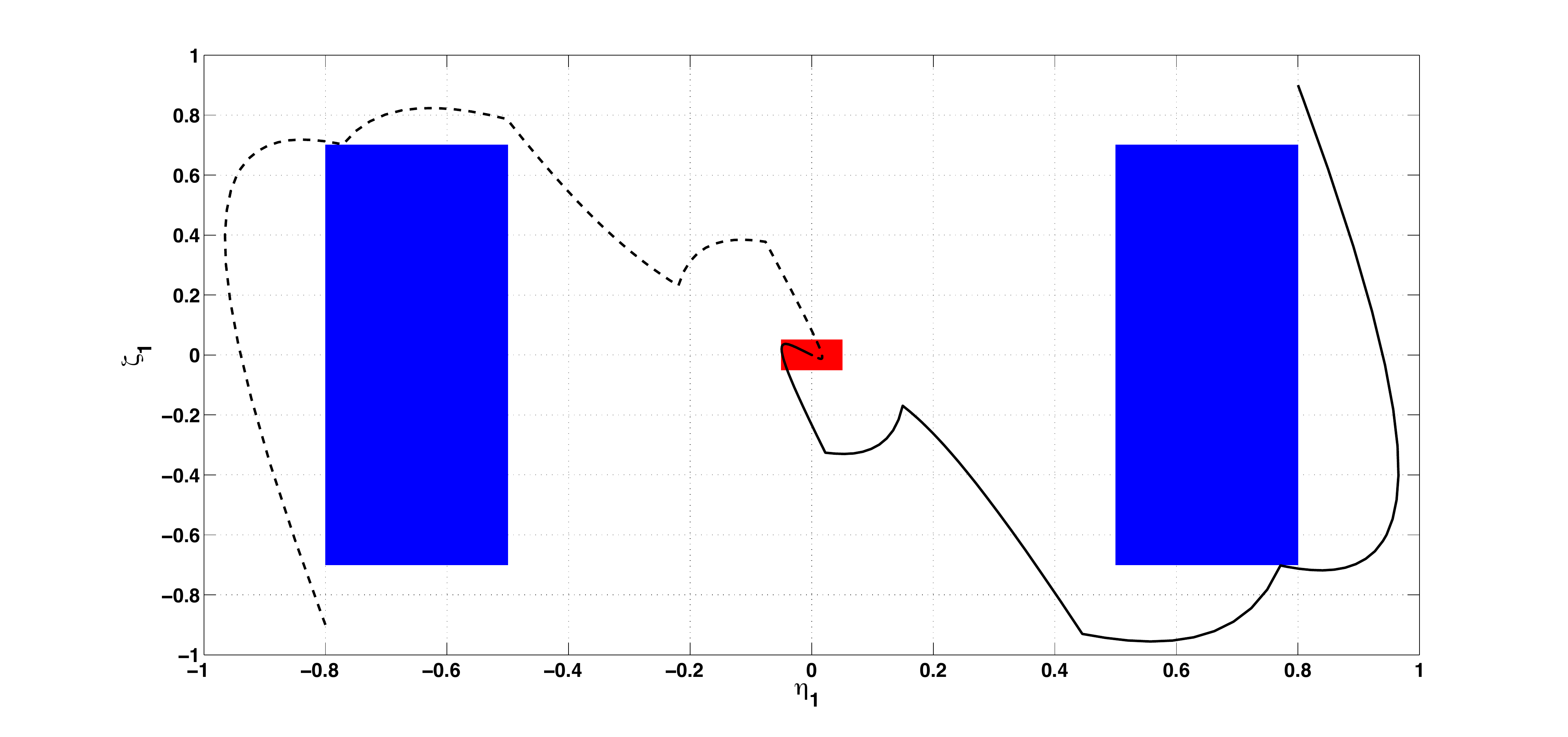}
  \caption{Evolutions of the closed-loop system with initial conditions $(0.8,~0.9)$, and $(-0.8,~-0.9)$.}
\label{fig1}
\end{figure*}

\begin{figure*}[t]
  \centering\includegraphics[width=15cm]{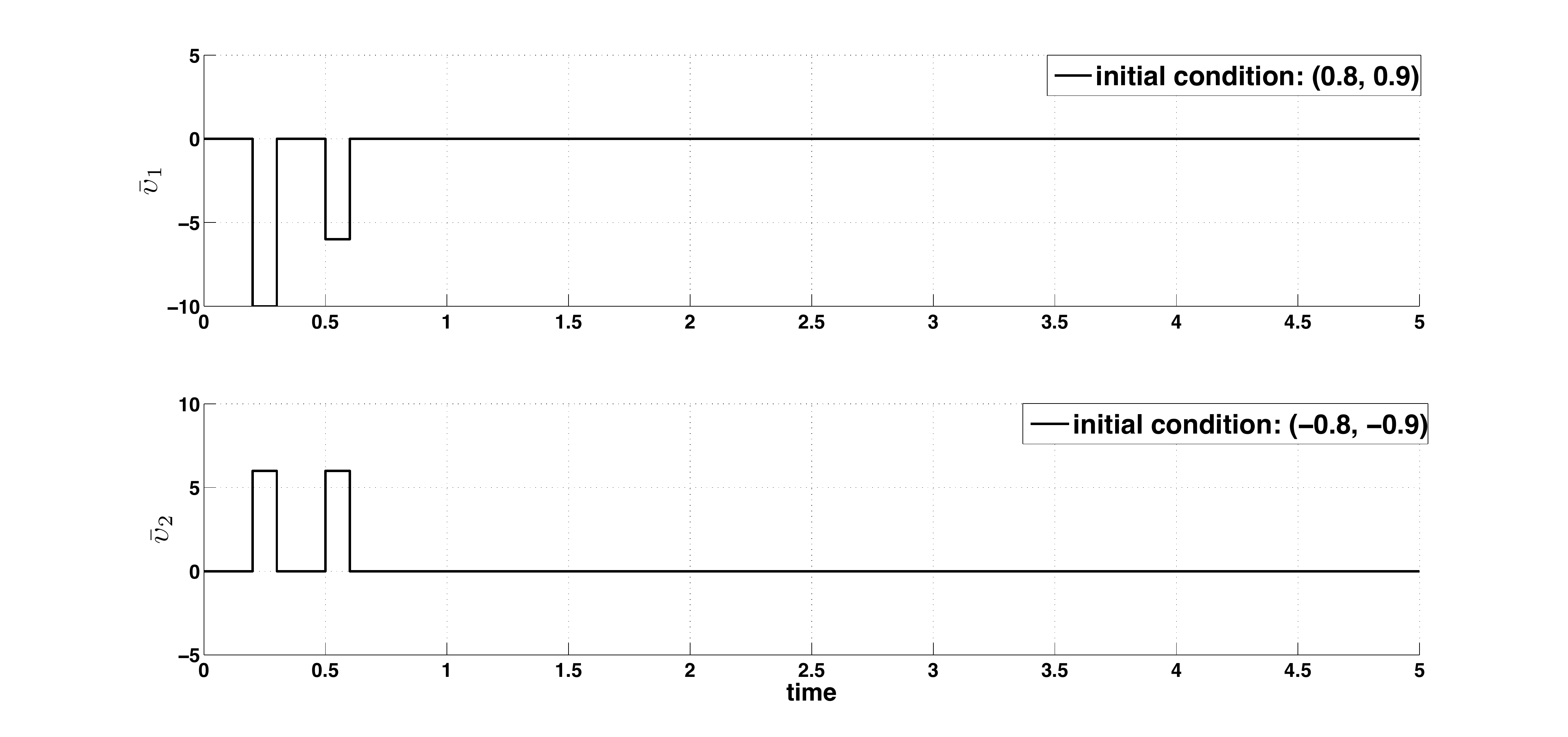}
  \caption{Evolutions of the input signals $\bar\upsilon_1$ and $\bar\upsilon_2$, corresponding to initial conditions $(0.8,~0.9)$, and $(-0.8,~-0.9)$, respectively.}
\label{fig2}
\end{figure*}

Resuming, the explicit availability of an incremental Lyapunov function let us to use the results in \cite{girard2} to construct a finite abstraction $S(\Sigma)$ for the control system $\Sigma$ in (\ref{example}), equipped with the control input in (\ref{law5}). This finite abstraction allowed us to use tools developed in computer since to synthesize a controller satisfying some logic specifications difficult to enforce using conventional controller synthesis methods. 
 
\section{Discussion}
In this paper we provided the characterizations of incremental stability, defined in \cite{majid1}, in terms of existence of incremental Lyapunov functions, defined in \cite{majid4}. We also provided sufficient conditions for incremental stability in terms of contraction metrics. Moreover, we developed a backstepping procedure to design controllers enforcing incremental input-to-state stability (or contraction properties) for the resulting closed-loop system. The proposed approach in this paper generalizes the work in \cite{jouffroy1,sharma,sharma1,majid1,majid4} by being applicable to larger classes of (not necessarily smooth) control systems and the work in \cite{pavlov} by enforcing incremental input-to-state stability rather than input-to-state convergence. Moreover, in contrast to the proposed backstepping design approach in \cite{pavlov}, here we provided a way of constructing incremental Lyapunov functions, which are known to be a key tool in the analysis provided in \cite{girard2,girard3,julius,kapinski}. As we showed in the example, the explicit existence of an incremental Lyapunov function helps us to use the results in \cite{girard2} to construct a finite bisimilar abstraction for a resulting incrementally stable closed-loop (non-smooth) control system.

\section{Acknowledgements}
The authors would like to thank Paulo Tabuada, Giordano Pola, and Manuel Mazo Jr. for their fruitful collaborations on the proof of Theorem \ref{theoremISS} and Lemma \ref{lemma20}. 

\bibliographystyle{alpha}
\bibliography{reference}
\end{document}